\documentclass[11pt, a4paper]{article}
\usepackage[T1]{fontenc}
\usepackage{amsfonts}
\usepackage{amsmath}
\usepackage{amssymb}
\usepackage{amsthm}
\usepackage{bbm}
\usepackage{bm}
\usepackage{mathrsfs}
\usepackage{verbatim}
\usepackage{setspace}
\usepackage{color}
\usepackage{pdfsync}
\usepackage{enumitem}

\theoremstyle{plain}
\newtheorem{theorem}{Theorem}[section]
\newtheorem{proposition}[theorem]{Proposition}
\newtheorem{lemma}[theorem]{Lemma}
\newtheorem{corollary}[theorem]{Corollary}

\theoremstyle{definition}
\newtheorem{definition}[theorem]{Definition}
\newtheorem{remark}[theorem]{Remark}

\newtheorem{assumption}[theorem]{Assumption}
\theoremstyle{remark}

\renewenvironment{thebibliography}[1]{%
\begin{oldthebibliography}{#1}%
\setlength{\baselineskip}{.9em}
\linespread{.9}
\small
\setlength{\parskip}{0ex}%
\setlength{\itemsep}{.1em}%
}%
{%
\end{oldthebibliography}%
}
\newcommand{\q}{\quad}

\newcommand{\eps}{\varepsilon}

\newcommand{\F}{\mathbb{F}}
\newcommand{\G}{\mathbb{G}}

\renewcommand{\L}{\mathbb{L}}

\newcommand{\N}{\mathbb{N}}

\newcommand{\Q}{\mathbb{Q}}
\newcommand{\R}{\mathbb{R}}

\newcommand{\cE}{\mathcal{E}}
\newcommand{\cF}{\mathcal{F}}
\newcommand{\cG}{\mathcal{G}}

\newcommand{\cN}{\mathcal{N}}

\newcommand{\cP}{\mathcal{P}}

\newcommand{\cU}{\mathcal{U}}

\newcommand{\bD}{\mathbf{D}}

\DeclareMathOperator{\UC}{\textstyle{UC}}
\DeclareMathOperator{\esssup}{ess\, sup}
\DeclareMathOperator{\essinf}{ess\, inf}

\newcommand{\as}{\mbox{-a.s.}}
\newcommand{\qs}{\mbox{-q.s.}}
\renewcommand{\ae}{\mbox{-a.e.}}
\newcommand{\1}{\mathbf{1}}

\newcommand{\br}[1]{\langle #1 \rangle}

\newcommand{\homega}{\hat{\omega}}
\newcommand{\tomega}{\tilde{\omega}}
\newcommand{\bomega}{\bar{\omega}}
\newcommand{\comega}{\check{\omega}}

\numberwithin{equation}{section}

\begin{document}

\title{\vspace{-0cm}
A Quasi-Sure Approach to the Control of Non-Markovian Stochastic Differential Equations
\date{First version: June 16, 2011. This version: March 18, 2012.}
\author{
  Marcel Nutz%
  \thanks{
  Dept.\ of Mathematics, Columbia University, New York, \texttt{mnutz@math.columbia.edu}
  }
 }
}
\maketitle \vspace{-0em}
\begin{abstract}
We study stochastic differential equations (SDEs) whose drift and diffusion coefficients are path-dependent and controlled. We construct a value process on the canonical path space, considered simultaneously under a family of singular measures, rather than the usual family of processes indexed by the controls. This value process is characterized by a second order backward SDE, which can be seen as a non-Markovian analogue of the Hamilton-Jacobi-Bellman partial differential equation. Moreover, our value process yields a generalization of the $G$-expectation to the context of SDEs.
\end{abstract}

\vspace{.5em}

{\small

}\\

\noindent \emph{Acknowledgements} Financial support by Swiss National Science Foundation
Grant PDFM2-120424/1 is gratefully acknowledged. The author thanks Shige Peng, Mete Soner, Nizar Touzi and Jianfeng Zhang for stimulating discussions and two anonymous referees for helpful comments.

\section{Introduction}\label{se:intro}

We consider a controlled stochastic differential equation (SDE) of the form
\begin{equation}\label{eq:SDEintro}
    X_t=x+\int_0^t\mu(r,X,\nu_r)\,dr+\int_0^t\sigma(r,X,\nu_r)\,dW_r,\quad 0\leq t\leq T,
\end{equation}
where $\nu$ is an adapted control process, $W$ is a Brownian motion and the Lipschitz-continuous coefficients $\mu(r,X,\nu_r)$ and $\sigma(r,X,\nu_r)$ may depend on the past trajectory $\{X_s, 0\leq s\leq r\}$ of the solution. Denoting by $X^\nu$ the solution corresponding to $\nu$, we are interested in the stochastic optimal control problem
\begin{equation}\label{eq:valueFctIntro}
  V_0:=\sup_\nu E[\xi(X^\nu)],
\end{equation}
where $\xi$ is a given functional. The standard approach to such a non-Markovian control problem (cf.\ \cite{ElKaroui.81, Elliott.82}) is to consider for each control $\nu$ the associated value process
\begin{equation}\label{eq:classicalValueProc}
  J^\nu_t = \mathop{\esssup}_{\tilde{\nu}:\; \tilde{\nu}=\nu \mbox{ on }[0,t]} E[\xi(X^{\tilde{\nu}})|\cF_t],
\end{equation}
where $(\cF_t)$ is the given filtration.
The dependence on $\nu$ reflects the presence of a \emph{forward} component in the optimization problem.

The situation is quite different in Markovian optimal control (cf.~\cite{FlemingSoner.06}), where one uses a single value function which depends on certain state variables but not on a control. This is essential to describe the value function by a differential equation, such as the Hamilton-Jacobi-Bellman PDE, which is the main merit of the dynamic programming approach. It is worth noting that this equation is always \emph{backward} in time.
An analogous description for~\eqref{eq:classicalValueProc} via backward SDEs (BSDEs, cf.\ \cite{PardouxPeng.90}) is available for certain popular problems such as utility maximization with power or exponential utility functions (e.g., \cite{HuImkellerMuller.05, Nutz.09b}) or drift control (e.g., \cite{Elliott.82}). However, this relies on a very particular algebraic structure which allows for a separation of $J^\nu$ into a backward part independent of $\nu$ and a forward part depending on $\nu$.

In this paper, we consider the problem~\eqref{eq:valueFctIntro} on the canonical space by recasting it as
\begin{equation}\label{eq:valueCanonIntro}
  V_0=\sup_\nu E^{P^\nu}[\xi(B)],
\end{equation}
where $P^\nu$ is the distribution of $X^\nu$ and $B$ is the canonical process, and we
describe its dynamic value by a single value process $V=\{V_t(\omega)\}$.
Formally, $V$ corresponds to a value function in the Markovian sense if we see the whole trajectory of the controlled system as a state variable. Even though~\eqref{eq:valueFctIntro} has features of coupled forward-backward type, the value process is defined in a purely backward manner: one may say that by constructing $V$ on the whole canonical space, we essentially calculate the value for all possible outcomes of the forward part. An important ingredient in the same vein is that $V$ is defined ``quasi-surely'' under the family of mutually singular measures $\{P^\nu\}$. Rather than forming a family of processes as in~\eqref{eq:classicalValueProc}, the necessary information is stored in a single process which is defined on a ``large'' part of the probability space; indeed, the process $V$ ``seen under $P^\nu$'' should be thought of as an analogue of $J^\nu$. Clearly, this is a necessary step to obtain a (second order) backward SDE. We remark that \cite{Peng.04} considered the same control problem~\eqref{eq:valueFctIntro} and also made a connection to nonlinear expectations. However, in \cite{Peng.04}, the value process was considered only under the given probability measure.

We first consider a fairly regular functional $\xi$ and define $V_t(\omega)$ as a conditional version of~\eqref{eq:valueCanonIntro}. Applying and advancing ideas from~\cite{SonerTouziZhang.2010aggreg} and \cite{Nutz.10Gexp}, regular conditional probability distributions are used
to define $V_t(\omega)$ for \emph{every} $\omega$ and prove a pathwise dynamic programming principle (Theorem~\ref{thm:DPP}). In a second step, we enlarge the class of functionals $\xi$ to an $L^1$-type space and prove that the value process admits a (quasi-sure) c\`adl\`ag modification (Theorem~\ref{th:modification}).

We also show that the value process falls into the class of sublinear expectations studied in~\cite{NutzSoner.10}. Indeed,
if $\xi$ is considered as a random variable on the canonical space, the mapping $\xi\mapsto V_t$ can be seen as a
generalization of the $G$-expectation~\cite{Peng.07, Peng.08}, which, by \cite{DenisHuPeng.2010}, corresponds to the case $\mu\equiv0$ and $\sigma(r,X,\nu_r)=\nu_r$, where the SDE~\eqref{eq:SDEintro} degenerates to a stochastic integral. Moreover, $V_t$
can be seen as a variant of the random $G$-expectation~\cite{Nutz.10Gexp}; cf.\ Remark~\ref{rk:randomG}.

Finally, we characterize $V$ by a second order backward SDE (2BSDE) in the spirit of~\cite{NutzSoner.10}; cf.\ Theorem~\ref{th:2bsde}. The second order is clearly necessary since the Hamilton-Jacobi-Bellman PDE for the Markovian case is fully nonlinear, while ordinary BSDEs correspond to semilinear equations.
2BSDEs were introduced in~\cite{CheriditoSonerTouziVictoir.07}, and in \cite{SonerTouziZhang.2010bsde} for the non-Markovian case. We refer to \cite{SonerTouziZhang.2010bsde} for the precise relation between 2BSDEs in the quasi-sure formulation and fully nonlinear parabolic PDEs.

We remark that our approach is quite different from the (backward) stochastic \emph{partial} differential equations studied in \cite{Peng.92a} for a similar control problem (mainly for uncontrolled volatility) and in \cite{LionsSouganidis.98a, LionsSouganidis.98b, BuckdahnMa.01a, BuckdahnMa.01b, BuckdahnMa.07} for so-called pathwise stochastic control problems. The relation to the path-dependent PDEs, introduced very recently in~\cite{Peng.11}, is yet to be explored.

The remainder of the paper is organized as follows. In Section~\ref{se:valueFct} we detail the controlled SDE and its conditional versions, define the value process for the case when $\xi$ is uniformly continuous and establish its regularity. The pathwise dynamic programming principle is proved in Section~\ref{se:dynamicProg}. In Section~\ref{se:extension} we extend the value process to a more general class of functionals $\xi$ and state its quasi-sure representation. The c\`adl\`ag modification is constructed in Section~\ref{se:modification}. In the concluding Section~\ref{se:2BSDE}, we provide the Hamilton-Jacobi-Bellman 2BSDE and interpret the value process as a variant of the random $G$-expectation.

\section{Construction of the Value Function}\label{se:valueFct}

In this section, we first introduce the setting and notation. Then, we define the value function $V_t(\omega)$ for a uniformly continuous reward functional $\xi$ and examine the regularity of $V_t$.

\subsection{Notation}

We fix a constant $T>0$ and let $\Omega:= C([0,T];\mathbbm{R}^d)$ be
the canonical space of continuous paths equipped with the uniform norm $\|\omega\|_T:=\sup_{0\leq s\leq T} |\omega_s|$, where
$|\cdot|$ is the Euclidean norm. We denote by $B$ the canonical process $B_t(\omega)=\omega_t$, by $P_0$ the Wiener measure, and by
$\F= \{\cF_t\}_{0\leq t\leq T}$ the (raw) filtration generated by $B$. Unless otherwise stated, probabilistic notions requiring a filtration (such as adaptedness) refer to $\F$.

For any probability measure $P$ on $\Omega$ and any $(t,\omega)\in [0,T]\times \Omega$, we can construct the corresponding regular conditional probability distribution $P^\omega_t$; cf.~\cite[Theorem~1.3.4]{StroockVaradhan.79}. We recall that $P^\omega_t$ is a probability kernel on $\cF_t\times\cF_T$; i.e., $P^\omega_t$ is a probability measure on $(\Omega,\cF_T)$ for fixed $\omega$ and $\omega\mapsto P^\omega_t(A)$ is $\cF_t$-measurable for each $A\in\cF_T$. Moreover, the expectation under $P^\omega_t$ is the conditional expectation under $P$:
\[
  E^{P^\omega_t}[\xi]=E^P[\xi|\cF_t](\omega)\quad P\as
\]
whenever $\xi$ is $\cF_T$-measurable and bounded. Finally, $P^\omega_t$ is concentrated on the set of paths
that coincide with $\omega$ up to $t$,
\begin{equation}\label{eq:measureConcentrated}
  P^\omega_t\big\{\omega'\in \Omega: \omega' = \omega \mbox{ on } [0,t]\big\} = 1.
\end{equation}
While $P^\omega_t$ is not defined uniquely by these properties, we choose and fix one version for each triplet $(t,\omega,P)$.

Let $t\in[0,T]$.
We denote by $\Omega^t:= \{\omega\in C([t,T];\R^d): \omega_t=0\}$
the shifted canonical space of paths starting at the origin.
For $\omega\in \Omega$, the shifted path $\omega^t\in \Omega^t$ is defined by
$\omega^t_r := \omega_r-\omega_t$ for $t\leq r\leq T$, so that $\Omega^t=\{\omega^t:\,\omega\in\Omega\}$.
Moreover, we denote by $P^t_0$ the Wiener measure on $\Omega^t$ and by $\F^t=\{\cF^t_r\}_{t\leq r\leq T}$ the (raw)
filtration generated by $B^t$, which can be identified with the canonical process on
$\Omega^t$.

Given two paths $\omega$ and $\tomega$, their concatenation at $t$ is the (continuous) path defined by
\[
 (\omega\otimes_t \tomega)_r := \omega_r \1_{[0,t)}(r) + (\omega_t + \tomega_r^t) \1_{[t, T]}(r),\quad 0\leq r\leq T.
\]
Given an $\cF_{T}$-measurable random variable $\xi$ on $\Omega$ and $\omega\in \Omega$,
we define the conditioned random variable $\xi^{t,\omega}$ on $\Omega$ by
\[
  \xi^{t, \omega}(\tomega) :=\xi(\omega\otimes_t \tomega),\q \tomega\in\Omega.
\]
Note that $\xi^{t, \omega}(\tomega)=\xi^{t, \omega}(\tomega^t)$; in particular, $\xi^{t, \omega}$ can also be seen as a random variable on $\Omega^t$. Then $\tomega\mapsto \xi^{t, \omega}(\tomega)$ is $\cF^t_{T}$-measurable and moreover, $\xi^{t,\omega}$ depends only on the restriction of $\omega$ to $[0,t]$.
We note that for an $\F$-progressively measurable process $\{X_r,\, r\in [s,T]\}$,
the conditioned process $\{X^{t, \omega}_r,\, r\in [t,T]\}$  is $\F^t$-progressively measurable.
If $P$ is a probability on $\Omega$, the measure $P^{t,\omega}$ on $\cF^t_T$ defined by
\[
  P^{t,\omega}(A):=P^\omega_t(\omega\otimes_t A),\quad A\in \cF^t_T, \quad\mbox{where }\omega\otimes_t A:=\{\omega\otimes_t \tomega:\, \tomega\in A\},
\]
is again a probability by~\eqref{eq:measureConcentrated}. We then have
\[
  E^{P^{t,\omega}}[\xi^{t,\omega}]=E^{P^\omega_t}[\xi] =E^P[\xi|\cF_t](\omega)\quad P\as
\]
Analogous notation will be used when $\xi$ is a random variable on $\Omega^s$ and $\omega\in\Omega^s$, where $0\leq s\leq t\leq T$.
We  denote by $\Omega^s_t:=\{\omega|_{[s,t]}:\,\omega\in\Omega^s\}$ the restriction of $\Omega^s$ to $[s,t]$, equipped with $\|\omega\|_{[s,t]}:=\sup_{r\in[s,t]}|\omega_r|$. Note that $\Omega^s_t$ can be identified with $\{\omega\in\Omega^s:\, \omega_r=\omega_t\mbox{ for }r\in [t,T]\}$. The meaning of $\Omega_t$ is analogous.

\subsection{The Controlled SDE}

Let $U$ be a nonempty Borel subset of $\R^m$ for some $m\in\N$.
We consider two given functions
\begin{equation*}%
  \mu: [0,T]\times\Omega\times U\to \R^d \quad\mbox{and}\quad\sigma: [0,T]\times\Omega\times U\to \R^{d\times d},
\end{equation*}
the drift and diffusion coefficients, such that
$(t,\omega)\mapsto \mu(t,X(\omega),\nu_t(\omega))$ and $(t,\omega)\mapsto \sigma(t,X(\omega),\nu_t(\omega))$ are progressively measurable
for any continuous adapted process $X$ and any $U$-valued progressively measurable process $\nu$.
In particular, $\mu(t,\omega,u)$ and $\sigma(t,\omega,u)$ depend only on the past trajectory $\{\omega_r,\,r\in [0,t]\}$, for any $u\in U$.
Moreover, we assume that there exists a constant $K>0$ such that
\begin{equation}\label{eq:LipschitzSigma}
  |\mu(t,\omega,u)-\mu(t,\omega',u)| + |\sigma(t,\omega,u)-\sigma(t,\omega',u)|\leq K\|\omega-\omega'\|_t
\end{equation}
for all $(t,\omega,\omega',u)\in [0,T]\times\Omega\times\Omega\times U$.
We denote by $\cU$ the set of all $U$-valued progressively measurable processes $\nu$ such that
\begin{equation}\label{eq:sigmaIntegrable}
  \int_0^T |\mu(r,X,\nu_r)|\,dr<\infty \quad\mbox{and}\quad \int_0^T |\sigma(r,X,\nu_r)|^2\,dr<\infty
\end{equation}
hold path-by-path for any continuous adapted process $X$.
Given $\nu\in\cU$, the stochastic differential equation
\begin{equation*}%
    X_t=x+\int_0^t\mu(r,X,\nu_r)\,dr+\int_0^t\sigma(r,X,\nu_r)\,dB_r,\quad 0\leq t\leq T \quad\mbox{under }P_0
\end{equation*}
has a $P_0$-a.s.\ unique strong solution for any initial condition $x\in\R^d$, which we denote by $X(0,x,\nu)$.
We shall denote by
\begin{equation}\label{eq:distribXNoShift}
  \bar{P}(0,x,\nu) := P_0 \circ X(0,x,\nu)^{-1}
\end{equation}
the distribution of $X(0,x,\nu)$ on $\Omega$ and by
\[
  P(0,x,\nu) := P_0 \circ \big(X(0,x,\nu)^0\big)^{-1}
\]
the distribution of $X(0,x,\nu)^0\equiv X(0,x,\nu)-x$; i.e., the solution which is translated to start at the origin. Note that
$P(0,x,\nu)$ is concentrated on $\Omega^0$ and can therefore be seen as a probability measure on $\Omega^0$.

We shall work under the following nondegeneracy condition.

\begin{assumption}\label{ass:filtrGen}
  Throughout this paper, we assume that
  \begin{equation}\label{eq:filtrGen}
    \overline{\F^X}^{P_0} \supseteq \F \quad \mbox{for all}\quad  X=X(0,x,\nu),
  \end{equation}
  where $\overline{\F^X}^{P_0}$ is the $P_0$-augmentation of the filtration generated by $X$ and $(x,\nu)$ varies over $\R^d\times\cU$.
\end{assumption}

One can construct situations where Assumption~\ref{ass:filtrGen} fails. For example, if $x=0$, $\sigma\equiv1$ and $\mu(r,X,\nu_r)=\nu_r$, then~\eqref{eq:filtrGen} fails for a suitable choice of $\nu$; see, e.g.,~\cite{FeldmanSmorodinsky.97}. The following is a positive result which covers many applications.

\begin{remark}
  {\textrm
  Let $\sigma$ be strictly positive definite and assume that
  $\mu(r,X,\nu_r)$ is a progressively measurable functional of $X$ and $\sigma(r,X,\nu_r)$. Then~\eqref{eq:filtrGen} holds true.

  Note that the latter assumption is satisfied in particular when $\mu$ is uncontrolled; i.e., $\mu(r,\omega,u)=\mu(r,\omega)$.
  }
\end{remark}

\begin{proof}
  Let $X=X(0,x,\nu)$. As the quadratic variation of $X$, the process
  $\int\sigma\sigma^\top(r,X,\nu_r)\,dr$ is adapted to the filtration generated by $X$. In view of our assumptions, it follows that
  \[
    M:=\int\sigma(r,X,\nu_r)\,dB_r  = X-x-\int\mu(r,X,\nu_t)\,dr
  \]
  has the same property. Hence $B=\int \sigma(r,X,\nu_r)^{-1}\,dM_r$ is again adapted to the filtration generated by $X$.
\end{proof}

\begin{remark}
  {\textrm
  For some applications, in particular when the SDE is of geometric form, requiring~\eqref{eq:filtrGen} to hold for all $x\in\R^d$ is too strong. One can instead fix the initial condition $x$ throughout the paper, then it suffices to require~\eqref{eq:filtrGen} only for that $x$.
  }
\end{remark}

Next, we introduce for fixed $t\in[0,T]$ an SDE on $[t,T]\times \Omega^t$ induced by $\mu$ and $\sigma$. Of course, the second argument of $\mu$ and $\sigma$ requires a path on $[0,T]$, so that it is necessary to specify a ``history'' for the SDE on $[0,t]$. This role is played by an arbitrary path $\eta\in\Omega$. Given $\eta$, we define the conditioned coefficients
\begin{align*}
  &\mu^{t,\eta}: [0,T]\times \Omega^t\times U \to \R^d, &\mu^{t,\eta}(r,\omega,u):=\mu(r,\eta\otimes_t \omega,u),\\
  &\sigma^{t,\eta}: [0,T]\times \Omega^t\times U \to \R^{d\times d}, &\sigma^{t,\eta}(r,\omega,u):=\sigma(r,\eta\otimes_t \omega,u).
\end{align*}
(More precisely, these functions are defined also when $\omega$ is a path not necessarily starting at the origin, but clearly their value at $(r,\omega,u)$ depends only on $\omega^t$.)
We observe that the Lipschitz condition~\eqref{eq:LipschitzSigma} is inherited; indeed,
\begin{align*}%
  |\mu^{t,\eta}(r,\omega,u)-\mu^{t,\eta}(r,\omega',u)|+&|\sigma^{t,\eta}(r,\omega,u)-\sigma^{t,\eta}(r,\omega',u)|\\
  & \leq K\|\eta\otimes_t\omega-\eta\otimes_t\omega'\|_{[t,r]}\\
  & = K\|\omega^t-\omega'^t\|_{[t,r]} \\
  & \leq 2 K \|\omega-\omega'\|_{[t,r]}.
\end{align*}
We denote by $\cU^t$ the set of all $\F^t$-progressively measurable, $U$-valued processes $\nu$
such that $\int_t^T |\mu(r,X,\nu_r)|\,dr<\infty$ and $\int_t^T |\sigma(r,X,\nu_r)|^2\,dr<\infty$
 for any continuous $\F^t$-adapted process $X=\{X_r, r\in[0,T]\}$.
For $\nu\in\cU^t$, the SDE
\begin{equation}\label{eq:SDEt}
    X_s=\eta_t+\int_t^s\mu^{t,\eta}(r,X,\nu_r)\,dr+\int_t^s\sigma^{t,\eta}(r,X,\nu_r)\,dB^t_r,\quad t\leq s \leq T \quad\mbox{under }P^t_0
\end{equation}
has a unique solution $X(t,\eta,\nu)$ on $[t,T]$. Similarly as above, we define
\[
  P(t,\eta,\nu) := P^t_0 \circ \big(X(t,\eta,\nu)^t\big)^{-1}
\]
to be the distribution of $X(t,\eta,\nu)^t\equiv X(t,\eta,\nu)-\eta_t$ on $\Omega^t$. Note that this is consistent with the notation $P(0,x,\nu)$ if $x$ is seen as a constant path.

\subsection{The Value Function}

We can now define the value function for the case when the reward functional $\xi$ is an element of
$\UC_b(\Omega)$, the space of bounded uniformly continuous functions on $\Omega$.

\begin{definition}\label{def:valueFct}
  Given $t\in[0,T]$ and $\xi\in\UC_b(\Omega)$, we define the value function
  \begin{equation}\label{eq:defV}
    V_t(\omega)=V_t(\xi;\omega)=\sup_{\nu\in\cU^t} E^{P(t,\omega,\nu)}[\xi^{t,\omega}],\quad (t,\omega)\in [0,T]\times \Omega.
  \end{equation}
\end{definition}

The function $\xi$ is fixed throughout Sections~\ref{se:valueFct} and~\ref{se:dynamicProg} and hence often suppressed in the notation. In view of the double dependence on $\omega$ in~\eqref{eq:defV}, the measurability of $V_t$ is not obvious. We have the following regularity result.

\begin{proposition}\label{pr:valueFunctionCont}
  Let $t\in[0,T]$ and $\xi\in\UC_b(\Omega)$. Then $V_t\in\UC_b(\Omega_t)$ and in particular $V_t$ is $\cF_t$-measurable.
  More precisely,
  \[
    |V_t(\omega)-V_t(\omega')| \leq \rho(\|\omega-\omega'\|_t)\quad \mbox{for all}\quad\omega,\omega'\in \Omega
  \]
  with a modulus of continuity $\rho$ depending only on $\xi$, the Lipschitz constant $K$ and the time horizon $T$.
\end{proposition}

The first source of regularity for $V_t$ is our assumption that $\xi$ is uniformly continuous; the second one is the Lipschitz property of the SDE. Before stating the proof of the proposition,  we examine the latter aspect in detail.

\begin{lemma}\label{le:weakContinuity}
  Let $\psi\in\UC_b(\Omega^t)$. There exists a modulus of continuity $\rho_{K,T,\psi}$,
  depending only on $K,T$ and the minimal modulus of continuity of $\psi$, such that
  \[
    \big|E^{P(t,\omega,\nu)}[\psi]-E^{P(t,\omega',\nu)}[\psi]\big| \leq \rho_{K,T,\psi}\big(\|\omega-\omega'\|_t\big)
  \]
  for all $t\in[0,T]$, $\nu\in\cU^t$ and $\omega,\omega'\in\Omega$.
\end{lemma}

\begin{proof}
  We set $E[\,\cdot\,]:=E^{P_0^t}[\,\cdot\,]$ to alleviate the notation. Let $\omega,\bomega\in\Omega$, set $X:=X(t,\omega,\nu)$ and
  $\bar{X}:=X(t,\bomega,\nu)$, and recall that $X^t=X-X_t=X-\omega_t$ and similarly $\bar{X}^t=\bar{X}-\bomega_t$.

  (i) We begin with a standard SDE estimate. Let
  $X^t=M^t+A^t$ and $\bar{X}^t=\bar{M}^t+\bar{A}^t$ be the semimartingale decompositions and
  $t\leq \tau \leq T$ be a stopping time such that $M^t,\bar{M}^t,A^t,\bar{A}^t$ are bounded on $[t,\tau]$. Then It\^o's formula and the Lipschitz property~\eqref{eq:LipschitzSigma} of $\sigma$ yield that
  \begin{align*}
    E[|M^t_\tau-\bar{M}^t_\tau|^2]
    & \leq E\int_t^\tau |\sigma(r,\omega\otimes_t X,\nu_r)-\sigma(r,\bomega\otimes_t \bar{X},\nu_r)|^2\,dr \\
    &\leq K^2 E\int_t^\tau \|\omega\otimes_t X - \bomega\otimes_t \bar{X}\|_r^2\,dr \\
    &\leq K^2 E\int_t^\tau \big(\|\omega-\bomega\|_t + \|X^t-\bar{X}^t\|_{[t,r]}\big)^2\,dr \\
    &\leq 2K^2T \|\omega-\bomega\|_t^2 + 2 K^2 \int_t^T E\big[\|X^t-\bar{X}^t\|_{[t,r\wedge \tau]}^2\big]\,dr.
  \end{align*}
  Hence, Doob's maximal inequality implies that
  \begin{align*}
     E\big[\|M^t-\bar{M}^t\|_{[t,\tau]}^2\big]
    & \leq 4 E[|M^t_\tau-\bar{M}^t_\tau|^2] \\
    &\leq 8K^2T \|\omega-\bomega\|_t^2 + 8 K^2 \int_t^T E \big[\|X^t-\bar{X}^t\|_{[t,r\wedge \tau]}^2\big]\,dr.
  \end{align*}
  Moreover, using the Lipschitz property~\eqref{eq:LipschitzSigma} of $\mu$, we also have that
  \begin{align*}
    |A^t_s-\bar{A}^t_s|
    &\leq \int_t^s |\mu(r,\omega\otimes_t X,\nu_r)-\mu(r,\bomega\otimes_t \bar{X},\nu_r)|\,dr \\
    &\leq K \int_t^s \big(\|\omega-\bomega\|_t + \|X^t-\bar{X}^t\|_{[t,r]}\big)\,dr
  \end{align*}
  for all $t\leq s\leq T$ and then Jensen's inequality yields that
  \[
    E\big[\|A^t-\bar{A}^t\|_{[t,\tau]}^2\big] \leq 2K^2 T^2 \|\omega-\bomega\|_t^2 + 2 K^2T \int_t^T E\big[\|X^t-\bar{X}^t\|_{[t,r\wedge \tau]}^2\big]\,dr.
  \]
  Hence, we have shown that
  \begin{align*}
     E\big[\|X^t-\bar{X}^t\|_{[t,\tau]}^2\big]
    &\leq C_0 \|\omega-\bomega\|_t^2 + C_0 \int_t^T E \big[\|X^t-\bar{X}^t\|_{[t,r\wedge \tau]}^2\big]\,dr,
  \end{align*}
  where $C_0$ depends only on $K$ and $T$,
  and we conclude by Gronwall's lemma that
  \[
    E\big[\|X^t-\bar{X}^t\|_{[t,\tau]}^2\big] \leq C  \|\omega-\bomega\|_t^2,\quad C:=C_0 e^{C_0T}.
  \]
  By the continuity of their sample paths, there exists a localizing sequence $(\tau_n)_{n\geq1}$ of stopping times such that $M^t,\bar{M}^t,A^t,\bar{A}^t$ are bounded on $[t,\tau_n]$ for each $n$.
  Therefore, monotone convergence and the previous inequality yield that
  \begin{equation}\label{eq:Gronwall}
    E\big[\|X^t-\bar{X}^t\|_{[t,T]}^2\big] \leq C  \|\omega-\bomega\|_t^2.
  \end{equation}

  (ii) Let $\tilde{\rho}$ be the minimal (nondecreasing) modulus of continuity for $\psi$,
  \[
    \tilde{\rho}(z):=\sup \big\{|\psi(\tomega)-\psi(\tomega')|:\;\tomega,\tomega'\in \Omega^t,\, \|\tomega-\tomega'\|_{[t,T]}\leq z\big\},
  \]
  and let $\rho$ be the concave hull of $\tilde{\rho}$.
  Then $\rho$ is a bounded continuous function satisfying $\rho(0)=0$ and $\rho\geq\tilde{\rho}$.
  Let $P:=P(t,\omega,\nu)$ and $\bar{P}:=P(t,\bomega,\nu)$, then $P$ and $\bar{P}$ are the distributions of $X^t$ and $\bar{X}^t$, respectively; therefore,
  \begin{equation}\label{eq:proofWeakCont}
    \big|E^P[\psi]-E^{\bar{P}}[\psi]\big|
       = \big| E[\psi(X^t) - \psi(\bar{X}^t)]\big|
       \leq E\big[\rho\big(\|X^t-\bar{X}^t\|_{[t,T]}\big)\big].
  \end{equation}
  Moreover, Jensen's inequality and~\eqref{eq:Gronwall} yield that
  \begin{align*}
    E\big[\rho\big(\|X^t-\bar{X}^t\|_{[t,T]}\big)\big]
        &\leq \rho \big(E\big[\|X^t-\bar{X}^t\|_{[t,T]}\big]\big) \\
        &\leq \rho \big(E\big[\|X^t-\bar{X}^t\|^2_{[t,T]}\big]^{1/2}\big)\\
        &\leq \rho \big(\sqrt{C}\|\omega-\bomega\|_t\big)
  \end{align*}
  for every $n$. In view of~\eqref{eq:proofWeakCont}, we have $|E^P[\psi]-E^{\bar{P}}[\psi]| \leq  \rho(\sqrt{C}\|\omega-\bomega\|_t)$; i.e., the result holds for $\rho_{K,T,\psi}(z):=\rho(\sqrt{C}z)$.
\end{proof}

After these preparations, we can prove the continuity of $V_t$.

\begin{proof}[Proof of Proposition~\ref{pr:valueFunctionCont}]
  To disentangle the double dependence on $\omega$ in~\eqref{eq:defV}, we first consider the function
  \[%
    (\eta,\omega)\mapsto E^{P(t,\eta,\nu)}[\xi^{t,\omega}],\quad (\eta,\omega)\in \Omega\times \Omega.
  \]
  Since $\xi\in\UC_b(\Omega)$, there exists a modulus of continuity $\rho^{(\xi)}$ for $\xi$; i.e.,
  \begin{equation*}%
    |\xi(\omega)-\xi(\omega')|\leq \rho^{(\xi)}(\|\omega-\omega'\|_T),\quad \omega,\omega'\in \Omega.
  \end{equation*}
  Therefore, we have for all $\tomega\in\Omega^t$ that
  \begin{align}\label{eq:modulusXi}
    |\xi^{t,\omega}(\tomega)-\xi^{t,\omega'}(\tomega)|
    &= |\xi(\omega\otimes_t\tomega)-\xi(\omega'\otimes_t\tomega)|\nonumber\\
    &\leq \rho^{(\xi)}(\|\omega\otimes_t\tomega-\omega'\otimes_t\tomega\|_T)\nonumber\\
    &= \rho^{(\xi)}(\|\omega-\omega'\|_t).
  \end{align}
  For fixed $\eta\in\Omega$, it follows that
  \begin{equation}\label{eq:omegaContVhat}
    \big|E^{P(t,\eta,\nu)}[\xi^{t,\omega}]-E^{P(t,\eta,\nu)}[\xi^{t,\omega'}]\big|
    \leq \rho^{(\xi)}(\|\omega-\omega'\|_t).
  \end{equation}

  Fix $\omega\in\Omega$ and let $\psi_\omega:=\xi^{t,\omega}$. Then $\rho^{(\xi)}$ yields a modulus of continuity for $\psi_\omega$; in particular, this modulus of continuity is uniform in $\omega$. Thus Lemma~\ref{le:weakContinuity} implies that the mapping $\eta\mapsto E^{P(t,\eta,\nu)}[\psi_\omega]$ admits a modulus of continuity $\rho_{T,K,\xi}$ depending only on $T,K,\xi$. In view of~\eqref{eq:omegaContVhat}, we conclude that
  \begin{equation*}%
    \big|E^{P(t,\eta,\nu)}[\xi^{t,\omega}]-E^{P(t,\eta',\nu)}[\xi^{t,\omega'}]\big| \leq \rho^{(\xi)}(\|\omega-\omega'\|_t) + \rho_{T,K,\xi}(\|\eta-\eta'\|_t)
  \end{equation*}
  for all $\eta,\eta',\omega,\omega'\in \Omega$ and in particular that
  \begin{equation}\label{eq:rewardUC}
    \big|E^{P(t,\omega,\nu)}[\xi^{t,\omega}]-E^{P(t,\omega',\nu)}[\xi^{t,\omega'}]\big| \leq \rho(\|\omega-\omega'\|_t),\quad \rho:=\rho^{(\xi)} + \rho_{T,K,\xi}
  \end{equation}
  for all $\omega,\omega'\in \Omega$. Passing to the supremum over $\nu\in\cU^t$, we obtain that
  $|V_t(\omega)-V_t(\omega')| \leq \rho(\|\omega-\omega'\|_t)$ for all $\omega,\omega'\in \Omega$,
  which was the claim.
\end{proof}

\section{Pathwise Dynamic Programming}\label{se:dynamicProg}

In this section, we provide a pathwise dynamic programming principle which is fundamental for the subsequent sections.
As we are working in the weak formulation~\eqref{eq:valueCanonIntro}, the arguments used here are similar to, e.g., \cite{SonerTouziZhang.2010dual}, while
\cite{Peng.04} gives a related construction in the strong formulation (i.e., working only under $P_0$).

We assume in this section that the following conditional version of Assumption~\ref{ass:filtrGen} holds true; however, we shall see later (Lemma~\ref{le:filtrGenFollows}) that this extended assumption holds automatically outside certain nullsets.

\begin{assumption}\label{ass:filtrGenCond}
  Throughout Section~\ref{se:dynamicProg}, we assume that
  \begin{equation}\label{eq:filtrGenCond}
    \overline{\F^X}^{P^t_0} \supseteq \F^t\quad \mbox{for}\quad  X:=X(t,\eta,\nu),
  \end{equation}
  for all $(t,\eta,\nu)\in [0,T]\times\Omega\times\cU^t$.
\end{assumption}

The main result of this section is the following dynamic programming principle. We shall also provide more general, quasi-sure versions of this result later (the final form being Theorem~\ref{th:DPPstop}).

\begin{theorem}\label{thm:DPP}
  Let $0\leq s\leq t\leq T$, $\xi\in\UC_b(\Omega)$ and set $V_r(\cdot)=V_r(\xi;\cdot)$. Then
  \begin{equation}\label{eq:DPP}
    V_s(\omega) = \sup_{\nu\in\cU^s} E^{P(s,\omega,\nu)}\big[ (V_t)^{s,\omega} \big]\quad \mbox{for all}\quad \omega\in\Omega.
  \end{equation}
\end{theorem}

We remark that in view of Proposition~\ref{pr:valueFunctionCont}, we may see $\xi\mapsto V_r(\xi;\cdot)$ as a mapping $\UC_b(\Omega)\to \UC_b(\Omega)$ and recast~\eqref{eq:DPP} as the semigroup property
\begin{equation}\label{eq:DPPSemiGrp}
  V_s=V_s\circ V_t\quad\mbox{on}\quad \UC_b(\Omega)\quad\mbox{for all }\quad 0\leq s\leq t\leq T.
\end{equation}

Some auxiliary results are needed for the proof of Theorem~\ref{thm:DPP}, which is stated at the end of this section. We start with the (well known) observation that conditioning the solution of an SDE yields the solution of a suitably conditioned SDE.

\begin{lemma}\label{le:conditioningX}
  Let $0\leq s\leq t\leq T$, $\nu\in\cU^s$ and $\bomega\in\Omega$. If $\bar{X}:=X(s,\bomega,\nu)$, then
  \begin{equation*}%
    \bar{X}^{t,\omega} = X\big(t,\bomega\otimes_s \bar{X}(\omega),\nu^{t,\omega}\big) \quad P_0^t\as
  \end{equation*}
  for all $\omega\in\Omega^s$.
\end{lemma}

\begin{proof}
  Let $\omega\in\Omega^s$. Using the definition and the flow property of $\bar{X}$, we have
  \begin{align*}
    \bar{X}_r
     &= \bomega_s+ \int_s^r \mu^{s,\bomega}(u,\bar{X},\nu_u)\,du+ \int_s^r \sigma^{s,\bomega}(u,\bar{X},\nu_u)\,dB^s_u\\
     &= \bar{X}_t+ \int_t^r \mu(u,\bomega\otimes_s\bar{X},\nu_u)\,du+ \int_t^r \sigma(u,\bomega\otimes_s\bar{X},\nu_u)\,dB^s_u\quad P^s_0\as
  \end{align*}
  for all $r\in[t,T]$. Hence, using that $(P_0^s)^{t,\omega}=P_0^t$ by the $P_0^s$-independence of the increments of $B^s$,
  \begin{equation}\label{eq:proofConditioningX}
    \bar{X}_r^{t,\omega} = \bar{X}^{t,\omega}_t +\int_t^r \!\mu(u,\bomega\otimes_s \bar{X}^{t,\omega},\nu^{t,\omega}_u)\,du + \int_t^r \!\sigma(u,\bomega\otimes_s \bar{X}^{t,\omega},\nu^{t,\omega}_u)\,dB_u^t\quad\! P_0^t\as
  \end{equation}
  Since $\bar{X}$ is adapted, we have $\bar{X}^{t,\omega}(\cdot)=\bar{X}(\omega\otimes_t\cdot)=\bar{X}(\omega)$ on $[s,t]$ and in particular
  \[
    \bomega \otimes_s \bar{X}^{t,\omega}=\bomega \otimes_s \bar{X}(\omega) \otimes _t \bar{X}^{t,\omega} = \eta\otimes_t \bar{X}^{t,\omega},\quad\mbox{for}\quad \eta:=\bomega\otimes_s \bar{X}(\omega).
  \]
  Therefore, recalling that $\bar{X}_s=\bomega_s$,~\eqref{eq:proofConditioningX} can be stated as
  \[
    \bar{X}_r^{t,\omega} = \eta_t + \int_t^r \mu^{t,\eta}(u,\bar{X}^{t,\omega},\nu^{t,\omega})\,du+ \int_t^r \sigma^{t,\eta}(u,\bar{X}^{t,\omega},\nu^{t,\omega})\,dB_u^t\quad P_0^t\as;
  \]
  i.e., $\bar{X}^{t,\omega}$ solves the SDE~\eqref{eq:SDEt} for the parameters $(t,\eta,\nu^{t,\omega})$. Now the result follows by the uniqueness of the
  solution to this SDE.
\end{proof}

Given $t\in[0,T]$ and $\omega\in\Omega$, we define
\begin{equation}\label{eq:defMeasSet}
  \cP(t,\omega)=\big\{P(t,\omega,\nu):\, \nu\in\cU^t \big\}.
\end{equation}
These sets have the following invariance property.

\begin{lemma}\label{le:shiftedMeasure}
  Let $0\leq s\leq t\leq T$ and $\bomega\in\Omega$. If $P\in \cP(s,\bomega)$, then
  \[
    P^{t,\omega} \in\cP(t,\bomega\otimes_s\omega)\quad \mbox{for}\quad P\ae\;\omega\in\Omega^s.
  \]
\end{lemma}

\begin{proof}
  Since $P\in \cP(s,\bomega)$, we have $P=P(s,\bomega,\nu)$ for some $\nu\in\cU^s$; i.e., setting $X:=X(s,\bomega,\nu)$, $P$ is
  the distribution of
  \[
    X^s = \int_s^\cdot \mu(r,\bomega\otimes_sX,\nu_r)\,dr+\int_s^\cdot \sigma(r,\bomega\otimes_sX,\nu_r)\,dB^s_r\quad\mbox{under}\quad P_0^s.
  \]
  We set $\hat{\mu}_r:=\mu(r,\bomega\otimes_s X,\nu_r)$ and $\hat{\sigma}_r:=\sigma(r,\bomega\otimes_s X,\nu_r)$ and see the above as the integral
  $\int_s^\cdot\hat{\mu}_r\,dr + \int_s^\cdot\hat{\sigma}_r\,dB^s_r$ rather than an SDE. As in~\cite[Lemma~2.2]{SonerTouziZhang.2010dual}, the nondegeneracy assumption
  \eqref{eq:filtrGenCond} implies the existence of a progressively measurable transformation $\beta_\nu: \Omega^s\to\Omega^s$ (depending on $s,\bomega,\nu$) such that
  \begin{equation}\label{eq:betaNu}
    \beta_\nu(X^s)=B^s \quad P_0^s\as
  \end{equation}
  Furthermore, a rather tedious calculation as in the proof of~\cite[Lemma~4.1]{SonerTouziZhang.2010dual} shows that
  \[
    P^{t,\omega} = P_0^t\circ \bigg(\int_s^\cdot\hat{\mu}^{t,\beta_\nu(\omega)}_r\,dr + \int_s^\cdot\hat{\sigma}^{t,\beta_\nu(\omega)}_r\,dB^t_r\bigg)^{-1}\quad \mbox{for}\quad P\ae\;\omega\in\Omega^s.
  \]
  Note that, abbreviating $\comega:=\beta_\nu(\omega)$, we have
  \[
     \hat{\mu}^{t,\beta_\nu(\omega)}_r=\mu\big(r,\bomega\otimes_s X^{t,\comega},\nu_r^{t,\comega}\big)=\mu\big(r,\bomega\otimes_s X(\comega)\otimes_t X^{t,\comega},\nu_r^{t,\comega}\big)
  \]
  and similarly for $\hat{\sigma}^{t,\beta_\nu(\omega)}$.
  Hence, we deduce by Lemma~\ref{le:conditioningX} that
  \[
    P^{t,\omega} = P(t,\bomega\otimes_s X(\comega),\nu^{t,\comega})\quad \mbox{for}\quad P\ae\;\omega\in\Omega^s.
  \]
  In view of~\eqref{eq:betaNu}, we have
  \begin{equation}\label{eq:betaalphaInverse}
    X^s(\beta_\nu(B^s))=B^s\quad P\as;
  \end{equation}
  i.e., $X(\comega)^s=X^s(\comega)=\omega$ for $P$-a.e.\ $\omega\in\Omega^s$, and we conclude that
  \begin{equation}\label{eq:shiftedMeasureFormula}
   P^{t,\omega} = P(t,\bomega\otimes_s \omega,\nu^{t,\comega})\quad \mbox{for}\quad P\ae\;\omega\in\Omega^s.
  \end{equation}
  In particular, $P^{t,\omega} \in\cP(t,\bomega\otimes_s\omega)$.
\end{proof}

\begin{lemma}[\textbf{Pasting}]\label{le:pasting}
  Let $0\leq s\leq t\leq T$, $\bomega\in\Omega$, $\nu\in\cU^s$ and set $P:=P(s,\bomega,\nu)$, $X:=X(s,\bomega,\nu)$. Let
  $(E^i)_{0\leq i\leq N}$ be a finite $\cF^s_t$-measurable partition of $\Omega^s$,
  $\nu^i\in\cU^t$ for $1\leq i\leq N$ and define $\bar{\nu}\in\cU^s$ by
  \begin{align*}
    \bar{\nu}(\omega)
      :=\1_{[s,t)} \nu(\omega)
      + \1_{[t,T]} \bigg[\nu(\omega)\1_{E^0}(X(\omega)^s) +\sum_{i=1}^N \nu^i(\omega^t)\1_{E^i}(X(\omega)^s)\bigg].%
  \end{align*}
  Then $\bar{P}:=P(s,\bomega,\bar{\nu})$ satisfies $\bar{P}=P$ on $\cF^s_t$ and
  \[
   \bar{P}^{t,\omega}=P(t,\bomega\otimes_s\omega, \nu^i)\quad\mbox{for}\quad P\ae\;\omega\in E^i,\quad1\leq i\leq N.
  \]
\end{lemma}

\begin{proof}
  As $\bar{\nu}=\nu$ on $[s,t)$, we have $X(s,\bomega,\bar{\nu})=X$ on $[s,t]$ and in particular $\bar{P}=P$ on $\cF^s_t$.
  Let $1\leq i\leq N$, we show that $\bar{P}^{t,\omega}=P(t,\bomega\otimes_s\omega, \nu^i)$ for $P$-a.e.\ $\omega\in E^i$.
  Recall from~\eqref{eq:shiftedMeasureFormula} that
  \[
    \bar{P}^{t,\omega}=P(t,\bomega\otimes_s\omega,\bar{\nu}^{t,\beta_{\bar{\nu}}(\omega)}) \quad \mbox{for}\quad \bar{P}\ae\;\omega\in\Omega^s,
  \]
  where $\beta_{\bar{\nu}}$ is defined as in~\eqref{eq:betaNu}. Since both sides of this equality depend only on the restriction of $\omega$ to $[s,t]$ and $X(s,\bomega,\bar{\nu})=X$ on $[s,t]$, we also have that
  \begin{equation}\label{eq:proofPastingShift}
    \bar{P}^{t,\omega}=P(t,\bomega\otimes_s\omega,\bar{\nu}^{t,\check{\omega}}) \quad \mbox{for}\quad P\ae\;\omega\in\Omega^s,
  \end{equation}
  where $\check{\omega}=\beta_\nu(\omega)$ is defined as below~\eqref{eq:betaNu}.
  Note that by~\eqref{eq:betaalphaInverse}, $\omega\in E^i$ implies $X(\check{\omega})^s\in E^i$ under $P$.
  (More precisely, if $A\subseteq \Omega^s$ is a set such that $A\subseteq E^i$ $P$-a.s., then
  $\{X(\check{\omega})^s:\, \omega\in A\}\subseteq E^i$ $P$-a.s.) In fact, since $X$ is adapted and $E^i\in\cF_t^s$, we even have that
  \[
    X(\check{\omega}\otimes_t\tomega )^s\in E^i\quad \mbox{for all}\quad \tomega\in\Omega^t, \quad \mbox{for}\quad P\ae\ \;\omega\in E^i.
  \]
  By the definition of $\bar{\nu}$, we conclude that
  \[
    \bar{\nu}^{t,\check{\omega}}(\tomega)=\bar{\nu}(\check{\omega}\otimes_t\tomega)=
    \nu^i((\check{\omega}\otimes_t\tomega)^t)=\nu^i(\tomega),\quad \tomega\in\Omega^t, \quad \mbox{for}\quad P\ae\ \;\omega\in E^i.
  \]
  In view of~\eqref{eq:proofPastingShift}, this yields the claim.
\end{proof}

\begin{remark}\label{rk:pastingMeasures}
  {\textrm
  In~\cite{SonerTouziZhang.2010dual} and \cite{Nutz.10Gexp}, it was possible to use a pasting of measures as follows: in the notation of Lemma~\ref{le:pasting}, it was possible to specify measures $P^i$ on $\Omega^t$, corresponding to certain admissible controls, and use the pasting
  $\hat{P}(A):= P(A\cap E^0) + \sum_{i=1}^N  E^P\big[P^i(A^{t,\omega})\1_{E^i}(\omega)\big]$
  to obtain a measure in $\Omega^s$ which again corresponded to some admissible control and satisfied
    \begin{equation}\label{eq:oldPasting}
    \hat{P}^{t,\omega}=P^i\quad\mbox{for}\quad \hat{P}\ae\;\omega\in E^i,
  \end{equation}
  which was then used in the proof of the dynamic programming principle.

  This is not possible in our SDE-driven setting. Indeed, suppose that $\hat{P}$ is of the form $P(s,\bomega,\nu)$ for some $\bomega$ and $\nu$, then we see from~\eqref{eq:shiftedMeasureFormula} that, when $\sigma$ is general,  $\hat{P}^{t,\omega}$ will depend explicitly on $\omega$, which contradicts~\eqref{eq:oldPasting}. Therefore, the subsequent proof uses an argument where~\eqref{eq:oldPasting} holds only at one specific $\omega\in E^i$; on the rest of $E^i$, we confine ourselves to controlling the error.
  }
\end{remark}

We can now show the dynamic programming principle. Apart from the difference remarked above, the basic pattern of the proof is the same as in \cite[Proposition~4.7]{SonerTouziZhang.2010dual}.

\begin{proof}[Proof of Theorem~\ref{thm:DPP}]
  Using the notation~\eqref{eq:defMeasSet}, our claim~\eqref{eq:DPP} can be stated as
  \begin{equation}\label{eq:DPPres}
    \sup_{P\in\cP(s,\bomega)} E^P\big[ \xi^{s,\bomega} \big] = \sup_{P\in\cP(s,\bomega)} E^P\big[ (V_t)^{s,\bomega} \big]\quad \mbox{for all}\quad  \bomega\in\Omega.
  \end{equation}

  (i)~We first show the inequality ``$\leq$'' in~\eqref{eq:DPPres}.
  Fix $\bomega\in \Omega$ and $P\in \cP(s,\bomega)$. Lemma~\ref{le:shiftedMeasure} shows that $P^{t,\omega}\in \cP(t,\bomega\otimes_s \omega)$ for $P$-a.e.\ $\omega\in \Omega^s$ and hence that
  \begin{align*}
   E^{P^{t,\omega}}\big[(\xi^{s,\bomega})^{t,\omega}\big]
    & = E^{P^{t,\omega}}\big[\xi^{t,\bomega \otimes_s\omega}\big]\\
    & \leq \sup_{P'\in\cP(t,\bomega\otimes_s \omega)} E^{P'}\big[\xi^{t,\bomega \otimes_s\omega}\big]\\
    & = V_t(\bomega\otimes_s \omega)\\
    & = V_t^{s,\bomega}(\omega)\quad\mbox{for $P$-a.e.\ $\omega\in \Omega^s$}.
  \end{align*}
  Since $V_t$ is measurable by Proposition~\ref{pr:valueFunctionCont}, we can take $P(d\omega)$-expectations on both sides to obtain that
  \[
    E^P\big[\xi^{s,\bomega}\big]
    = E^{P}\Big[E^{P^{t,\omega}}\big[(\xi^{s,\bomega})^{t,\omega}\big]\Big]
    \leq E^P\big[V_t^{s,\bomega}\big].
  \]
  We take the supremum over $P\in\cP(s,\bomega)$ on both sides and obtain the claim.

  (ii)~We now show the inequality ``$\geq$'' in~\eqref{eq:DPPres}. Fix $\bomega\in\Omega$, $\nu\in\cU^s$ and let $P=P(s,\bomega,\nu)$. We fix $\eps>0$ and construct a countable cover of the state space as follows.

  Let $\homega\in\Omega^s$. By the definition of $V_t(\bomega\otimes_s\homega)$, there exists $\nu^{(\homega)}\in\cU^t$ such that
  $P^{(\homega)}:=P(t,\bomega\otimes_s\homega,\nu^{(\homega)})$ satisfies
  \begin{equation}\label{eq:epsOptimizer}
    V_t(\bomega\otimes_s\homega)\leq E^{P^{(\homega)}}[\xi^{t,\bomega\otimes_s\homega}]+\eps.
  \end{equation}
  Let $B(\eps,\homega)\subseteq \Omega^s$ denote the open $\|\cdot\|_{[s,t]}$-ball of radius $\eps$ around $\homega$. Since
  $(\Omega^s,\|\cdot\|_{[s,t]})$ is a separable (quasi-)metric space and therefore Lindel\"of, there exists a
  sequence $(\homega^i)_{i\geq1}$ in $\Omega^s$ such that the balls $B^i:=B(\eps,\homega^i)$ form a cover of $\Omega^s$. As an $\|\cdot\|_{[s,t]}$-open set, each $B^i$ is $\cF^s_t$-measurable and hence
  \[
    E^1:=B^1,\quad E^{i+1}:=B^{i+1}\setminus (E^1 \cup\dots\cup E^i),\quad i\geq1
  \]
  defines a partition $(E^i)_{i\geq 1}$ of $\Omega^s$. Replacing $E^i$ by
  \[
    \big(E^i\cup\{\homega^i\}\big) \setminus \{\homega^j:\,j\geq 1,\,j\neq i\}
  \]
  if necessary, we may assume that $\homega^i\in E^i$ for $i\geq1$. We set $\nu^i:=\nu^{(\homega^i)}$ and $P^i:=P(t,\bomega\otimes_s\homega^i,\nu^i)$.

  Next, we paste the controls $\nu^i$. Fix $N\in\N$ and let $A_N:= E^1\cup\dots\cup E^N$, then $\{A_N^c,E^1,\dots,E^N\}$ is a finite partition of $\Omega^s$. Let $X:=X(s,\bomega,\nu)$, define
  \[
    \bar{\nu}(\omega)
    :=\1_{[s,t)}\nu(\omega)
      + \1_{[t,T]}\bigg[\nu(\omega)\1_{A_N^c}(X(\omega)^s) +\sum_{i=1}^N \nu^i(\omega^t)\1_{E^i}(X(\omega)^s)\bigg]
  \]
  and let $\bar{P}:=P(s,\bomega,\bar{\nu})$. Then, by Lemma~\ref{le:pasting}, we have
  $\bar{P}=P$ on $\cF^s_t$ and $\bar{P}^{t,\omega}=P(t,\bomega\otimes_s\omega, \nu^i)$ for all $\omega\in \tilde{E}^i$, for some subset
  $\tilde{E}^i\subseteq E^i$ of full measure $P$.
  Let us assume for the moment that
  \begin{equation}\label{eq:proofDPPprelimAssumption}
    \homega^i\in \tilde{E}^i\quad\mbox{for}\quad 1\leq i\leq N,
  \end{equation}
  then we may conclude that
  \begin{equation}\label{eq:shiftIsPi}
    \bar{P}^{t,\homega^i}=P^i\quad \mbox{for}\quad 1\leq i\leq N.
  \end{equation}

  Recall from Proposition~\ref{pr:valueFunctionCont} that $V_t$ admits a modulus of continuity $\rho^{(V_t)}$. Moreover, we obtain similarly as in~\eqref{eq:modulusXi} that there exists a modulus of continuity
  $\rho^{(\xi)}$ such that
  \[
   |\xi^{t,\bomega\otimes_s\omega}-\xi^{t,\bomega\otimes_s\omega'}|\leq\rho^{(\xi)}(\|\omega-\omega'\|_{[s,t]}).
  \]
  Let $\omega\in E^i\subseteq \Omega^s$ for some $1\leq i\leq N$, then $\|\omega-\homega^i\|_{[s,t]}< \eps$. Together with~\eqref{eq:epsOptimizer} and~\eqref{eq:shiftIsPi}, we obtain that
  \begin{align}
    V^{s,\bomega}_t(\omega)
     & \leq V^{s,\bomega}_t(\homega^i) + \rho^{(V_t)}(\eps) \nonumber\\
     & \leq E^{P^i}[\xi^{t,\bomega\otimes_s\homega^i}] + \eps  + \rho^{(V_t)}(\eps) \nonumber\\
     & =  E^{\bar{P}^{t,\homega^i}}[\xi^{t,\bomega\otimes_s\homega^i}] + \eps  + \rho^{(V_t)}(\eps). \label{eq:proofDPP1}
  \end{align}
  Recall from~\eqref{eq:rewardUC} that the mapping
  \[
    \omega'\mapsto E^{P(t,\omega',\nu^i)}[\xi^{t,\omega'}]
  \]
  is uniformly continuous with a modulus $\tilde{\rho}$ independent of $i$ and $N$. Since $\omega\in E^i$,
  it follows that
  \begin{align}\label{eq:proofDPP2}
   E^{\bar{P}^{t,\homega^i}}&[\xi^{t,\bomega\otimes_s\homega^i}]-E^{\bar{P}^{t,\omega}}[\xi^{t,\bomega\otimes_s\omega}] \nonumber\\
   &=E^{P(t,\bomega\otimes_s\homega^i,\nu^i)}[\xi^{t,\bomega\otimes_s\homega^i}]-E^{P(t,\bomega\otimes_s\omega,\nu^i)}[\xi^{t,\bomega\otimes_s\omega}] \nonumber \\
   & \leq \tilde{\rho}(\eps)\quad \mbox{for}\quad \bar{P}\ae\;\omega\in E^i.
  \end{align}
  Setting $\rho(\eps):=\tilde{\rho}(\eps)+ \eps  + \rho^{(V_t)}(\eps)$ and noting that
  \[
   E^{\bar{P}^{t,\omega}}[\xi^{t,\bomega\otimes_s\omega}] = E^{\bar{P}^{t,\omega}}[(\xi^{s,\bomega})^{t,\omega}] = E^{\bar{P}}[\xi^{s,\bomega}|\cF^s_t](\omega),
  \]
  the inequalities \eqref{eq:proofDPP1} and~\eqref{eq:proofDPP2} imply that
  \begin{equation}\label{eq:proofDPPConcatMeas}
    V^{s,\bomega}_t(\omega) \leq E^{\bar{P}}\big[\xi^{s,\bomega}\big|\cF^s_t\big](\omega) + \rho(\eps)
  \end{equation}
  for $\bar{P}$-a.e.\ (and thus $P$-a.e.) $\omega\in E^i$. This holds for all $1\leq i\leq N$.
  As $P=\bar{P}$ on $\cF^s_t$, taking $P$-expectations yields
  \begin{equation}\label{eq:proofDPP3}
   E^P[V^{s,\bomega}_t \1_{A_N}] \leq E^{\bar{P}_N}[\xi^{s,\bomega} \1_{A_N}] + \rho(\eps),
  \end{equation}
  where we write $\bar{P}_N=\bar{P}$ to recall the dependence on $N$.
  Since $A_N\uparrow \Omega^s$, we have $\bar{P}_N(A_N^c)=P(A_N^c)\to0$ as $N\to\infty$.
  In view of
  \[
    E^{\bar{P}_N}[\xi^{s,\bomega} \1_{A_N}]
    = E^{\bar{P}_N}[\xi^{s,\bomega}] - E^{\bar{P}_N}[\xi^{s,\bomega} \1_{A_N^c}]
    \leq E^{\bar{P}_N}[\xi^{s,\bomega}] + \|\xi\|_{\infty}P_N(A_N^c),
  \]
  we conclude from~\eqref{eq:proofDPP3} that
  \[
   E^P[V^{s,\bomega}_t]
   \leq \limsup_{N\to\infty}E^{\bar{P}_N}[\xi^{s,\bomega}] + \rho(\eps)
   \leq \sup_{P'\in\cP(s,\bomega)} E^{P'}[\xi^{s,\bomega}] + \rho(\eps).
  \]
  Since $P\in\cP(s,\bomega)$
  was arbitrary, letting $\eps\to0$ completes the proof of~\eqref{eq:DPPres}.

  It remains to argue that our assumption~\eqref{eq:proofDPPprelimAssumption} does not entail a loss of generality.
  Indeed, assume that $\homega^i\notin \tilde{E}^i$ for some $i$. Then there are two possible cases. The case $P(E^i)=0$ is easily seen to be harmless; recall that the measure $P$ was fixed throughout the proof. In the case $P(E^i)>0$, we also have
  $P(\tilde{E}^i)>0$ and in particular $\tilde{E}^i\neq \emptyset$. Thus we can replace $\homega^i$ by an arbitrary element of $\tilde{E}^i$ (which can be chosen independently of $N$). Using the continuity of the value function (Proposition~\ref{pr:valueFunctionCont}) and of the reward function~\eqref{eq:rewardUC}, we see that the above arguments still apply if we add an additional modulus of continuity in~\eqref{eq:proofDPP1}.
\end{proof}

\section{Extension of the Value Function}\label{se:extension}

In this section, we extend the value function $\xi\mapsto V_t(\xi;\cdot)$ to an $L^1$-type space of random variables $\xi$, in the spirit of, e.g., \cite{DenisMartini.06}. While the construction of $V_t$ in the previous section required a precise analysis ``$\omega$ by $\omega$'', we can now move towards a more probabilistic presentation. In particular, we shall often write $V_t(\xi)$ for the random variable $\omega\mapsto V_t(\xi;\omega)$.

For reasons explained in Remark~\ref{rk:xFixed} below, we fix from now on an initial condition $x\in\R^d$ and let
\[
  \cP_x:=\{P(0,x,\nu): \nu\in\cU\}
\]
be the corresponding set of measures at time $s=0$. Given a random variable $\psi$ on $\Omega$, we write $\psi^x$ as a shorthand for $\psi^{0,x}\equiv\psi(x\otimes_0\cdot)$. We also write $V^x_t(\xi)$ for $(V_t(\xi))^x$.

Given $p\in[1,\infty)$, we define $L^p_{\cP_x}$ to be the space of $\cF_T$-measurable random variables $X$ satisfying
\[
  \|X\|_{L^p_{\cP_x}}:=\sup_{P\in \cP_x} \|X\|_{L^p(P)}<\infty,
\]
where $\|X\|^p_{L^p(P)}:=E^P[|X|^p]$. More precisely, we identify functions which are equal $\cP_x$-quasi-surely, so that $L^p_{\cP_x}$ becomes a Banach space. (Two functions are equal $\cP_x$-quasi-surely, $\cP_x$-q.s.\ for short, if they are equal $P$-a.s.\ for all $P\in\cP_x$.) Furthermore, given $t\in[0,T]$,
\[
  \L^p_{\cP_x}(\cF_t)\mbox{ is defined as the $\|\cdot\|_{L^p_{\cP_x}}$-closure of }\UC_b(\Omega_t)\subseteq L^p_{\cP_x}.
\]
Since any $L^p_{\cP_x}$-convergent sequence has a $\cP_x$-q.s.\ convergent subsequence, any element of $\L^p_{\cP_x}(\cF_t)$ has an $\cF_t$-measurable representative.
For brevity, we shall often write $\L^p_{\cP_x}$ for $\L^p_{\cP_x}(\cF_T)$. %

\begin{remark}\label{rk:quasiUnifCont}
  {\textrm
  The space $\L^p_{\cP_x}$ can be described as follows.
  We say that $\xi\in L^p_{\cP_x}$ is $\cP_x$-quasi uniformly continuous if $\xi$ has a representative $\xi'$ with the property that for all $\eps>0$ there exists an open set $G\subseteq \Omega$ such that $P(G)<\eps$ for all $P\in\cP$ and such that the restriction $\xi'|_{\Omega\setminus G}$ is uniformly continuous.
  Then $\L^p_{\cP_x}$ consists of all $\xi\in L^p_{\cP_x}$ such that $\xi$ is $\cP_x$-quasi uniformly continuous and $\lim_{n\to\infty} \|\xi\1_{\{|\xi|\geq n\}}\|_{L^p_{\cP_x}}=0$.
  Moreover, If $\cP_x$ is weakly relatively compact, then $\L^p_{\cP_x}$ contains all
  bounded continuous functions on $\Omega$.

  The proof is the same as in~\cite[Proposition~5.2]{Nutz.10Gexp}, which, in turn, followed an argument of~\cite{DenisHuPeng.2010}.
  }
\end{remark}

Before extending the value function to $\L^1_{\cP_x}$, let us explain why we are working under a fixed initial condition $x\in\R^d$.

\begin{remark}\label{rk:xFixed}
  {\textrm
  There is no fundamental obstruction to writing the theory without fixing the initial condition $x$; in fact, most of the results would be more elegant if stated using $\bar{\cP}$ instead of $\cP_x$, where $\bar{\cP}$ is the set of all distributions of the form~\eqref{eq:distribXNoShift}, with arbitrary initial condition.
  However, the set $\bar{\cP}$ is very large and therefore the corresponding space
  $\L^1_{\bar{\cP}}$ is very small, which is undesirable for the domain of our extended value function. As an illustration, consider a random variable of the form $\xi(\omega):=f(\omega_0)$ on $\Omega$, where $f: \R^d\to\R$ is a measurable function. Then
  \[
    \|\xi\|_{L^1_{\bar{\cP}}}=\sup_{x\in\R^d} |f(x)|;
  \]
  i.e., $\xi$ is in $L^1_{\bar{\cP}}$ only when $f$ is uniformly bounded.
  As a second issue in the same vein, it follows from the Arzel\`a-Ascoli theorem that the set $\bar{\cP}$ is never weakly relatively compact. The latter property, which is satisfied by $\cP_x$ for example when $\mu$ and $\sigma$ are bounded, is sometimes useful in the context of quasi-sure analysis.
  }
\end{remark}

\begin{lemma}\label{le:LipschitzAndExtension}
  Let $p\in[1,\infty)$. The mapping $V^x_t$ on $\UC_b(\Omega)$ is 1-Lipschitz,
  \[
    \|V^x_t(\xi)-V^x_t(\psi)\|_{L^p_{\cP_x}}\leq \|\xi-\psi\|_{L^p_{\cP_x}}\quad\mbox{for all}\quad  \xi,\psi\in \UC_b(\Omega).
  \]
  As a consequence, $V^x_t$ uniquely extends to a Lipschitz-continuous mapping
  \[
    V^x_t:\, \L^p_{\cP_x}(\cF_T)\to \L^p_{\cP_x}(\cF_t).
  \]
\end{lemma}

\begin{proof}
 The argument is standard and included only for completeness. Note that $|\xi-\psi|^p$ is again in $\UC_b(\Omega)$. The definition of $V^x_t$ and Jensen's inequality imply that
 $|V^x_t(\xi)-V^x_t(\psi)|^p\leq V^x_t(|\xi-\psi|)^p\leq V^x_t(|\xi-\psi|^p)$. Therefore,
  \[
    \|V^x_t(\xi)-V^x_t(\psi)\|_{L^p_{\cP_x}}
      \leq \sup_{P\in \cP_x} E^P\big[V^x_t(|\xi-\psi|^p)\big]^{1/p}\\
      = \sup_{P\in \cP_x} E^P[ |\xi-\psi|^p ]^{1/p},
  \]
  where the equality is due to~\eqref{eq:DPP} applied with $s=0$. (For the case $s=0$, the additional Assumption~\ref{ass:filtrGenCond} was not used in the previous section.) Recalling from Proposition~\ref{pr:valueFunctionCont} that $V^x_t$ maps $\UC_b(\Omega)$ to $\UC_b(\Omega_t)$, it follows that the extension maps $\L^p_{\cP_x}$ to $\L^p_{\cP_x}(\cF_t)$.
\end{proof}

\subsection{Quasi-Sure Properties of the Extension}

In this section, we provide some auxiliary results of technical nature. The first one will (quasi-surely) allow us to appeal to the results in the previous section without imposing Assumption~\ref{ass:filtrGenCond}. This is desirable since we would like to end up with quasi-sure theorems whose statements do not involve regular conditional probability distributions.

\begin{lemma}\label{le:filtrGenFollows}
  Assumption~\ref{ass:filtrGen} implies that Assumption~\ref{ass:filtrGenCond} holds for $\cP_x$-quasi-every $\eta\in \Omega$ satisfying $\eta_0=x$.
\end{lemma}

For the proof of this lemma, we shall use the following result.

\begin{lemma}\label{le:filtrationConditioning}
  Let $Y$ and $Z$ be continuous adapted processes, $t\in[0,T]$ and let $P$ be a probability measure on $\Omega$. Then $\overline{\F^Y}^P\supseteq \F^Z$ implies that
  \[
    \overline{\F^{Y^{t,\omega}}}^{P^{t,\omega}}\supseteq \F^{Z^{t,\omega}}\quad\mbox{for}\quad P\ae\;\omega\in\Omega.
  \]
\end{lemma}

\begin{proof}
  The assumption implies that there exists a progressively measurable transformation $\beta: \Omega\to\Omega$ such that
  $Z=\beta(Y)$ $P$-a.s. For $P$-a.e.\ $\omega\in \Omega$, it follows that $Z(\omega\otimes_t\cdot )=\beta(Y(\omega\otimes_t\cdot))$ $P^{t,\omega}$-a.s., which, in turn, yields the result.
\end{proof}

\begin{proof}[Proof of Lemma~\ref{le:filtrGenFollows}]
  Let $X:=X(t,\eta,\nu)$ with $\eta_0=x$, we have to show that $\overline{\F^X}^{P^t_0} \supseteq \F^t$ whenever $\eta^0$ is outside some $\cP_x$-polar set. Hence we shall fix an arbitrary $\hat{P}\in\cP_x$ and show that the result holds on a set of full measure $\hat{P}$. %

  Let $\hat{P}\in\cP_x$, then $\hat{P}=P(0,x,\hat{\nu})$ for some $\hat{\nu}\in\cU$ and $\hat{P}$ is concentrated on the image of $\hat{X}^0$, where $\hat{X}:=X(0,x,\hat{\nu})$. That is, recalling that $\eta_0=\hat{X}_0=x$, we may assume that $\eta=\hat{X}(\omega)$ for some $\omega\in\Omega$.
  Let
  \begin{equation}\label{eq:proofFiltrGenFollows}
    \bar{\nu} (\tilde{\omega}) := \1_{[0,t)} \hat{\nu}(\tilde{\omega})+\1_{[t,T]}\nu(\tilde{\omega}^t).
  \end{equation}
  Then $\bar{X}:=X(0,x,\bar{\nu})$ satisfies $\bar{X}=\hat{X}$ on $[0,t]$ and hence we may assume that $\eta=\bar{X}(\omega)$ on $[0,t]$. Using Lemma~\ref{le:conditioningX} and~\eqref{eq:proofFiltrGenFollows}, we deduce that
  \[
    \bar{X}^{t,\omega}= X\big(t,x\otimes \bar{X}(\omega),\bar{\nu}^{t,\omega}\big)=X(t,\eta,\nu)=X.
  \]
  Since $\F\subseteq\overline{\F^{\bar{X}}}^{P_0}$ by Assumption~\ref{ass:filtrGen}, we conclude that
  \[
    \overline{\F^t}^{P^t_0}\subseteq \overline{\F^{\bar{X}^{t,\omega}}}^{P^t_0} = \overline{\F^X}^{P^t_0}
  \]
  by using Lemma~\ref{le:filtrationConditioning} with $Z$ being the canonical process.
\end{proof}

The next two results show that (for $\mu\equiv 0$ and $\sigma$ positive definite) the mapping $\xi\mapsto V^x_t(\xi)$ on $\L^1_{\cP_x}$ falls into the general class of sublinear expectations considered in~\cite{NutzSoner.10}, whose techniques we shall apply in the subsequent section. More precisely, the two lemmas below yield the validity of its main condition~\cite[Assumption~4.1]{NutzSoner.10}.

The following property is known as \emph{stability under pasting} and well known to be important in non-Markovian control. It should not be confused with the pasting discussed in Remark~\ref{rk:pastingMeasures}, where the considered measures correspond to different points in time.

\begin{lemma}\label{le:stabilityUnderPasting}
  Let $\tau$ be an $\F$-stopping time and let $\Lambda\in\cF_\tau$. Let $P,P^1,P^2\in\cP_x$ satisfy
  $P^1=P^2=P$ on $\cF_\tau$. Then
  \begin{equation*}%
    \bar{P}(A):= E^P\big[P^1(A|\cF_\tau)\1_{\Lambda} + P^2(A|\cF_\tau)\1_{\Lambda^c}\big],\quad A\in\cF_T
  \end{equation*}
  defines an element of $\cP_x$.
\end{lemma}

\begin{proof}
  It follows from the definition of the conditional expectation that $\bar{P}$ is a probability measure which is characterized by the properties
  \begin{equation}\label{eq:PastingCharact}
    \bar{P}=P\mbox{ on }\cF_\tau \quad\mbox{and}\quad \bar{P}^{\tau(\omega),\omega}=
    \begin{cases}
      (P^1)^{\tau(\omega),\omega} & \mbox{for}\quad P\ae\;\omega\in\Lambda, \\
      (P^2)^{\tau(\omega),\omega} & \mbox{for}\quad P\ae\;\omega\in\Lambda^c.
    \end{cases}
  \end{equation}
  Let $\nu,\nu^1,\nu^2\in\cU$ be such that $P=P(0,x,\nu)$ and $P^i=P(0,x,\nu^i)$ for $i=1,2$. Moreover, let $X:=X(0,x,\nu)$,
  define $\bar{\nu}\in\cU$ by
  \begin{align*}
    \bar{\nu}_r(\omega)
    &:=\1_{[0,\tau(X(\omega)^0))}(r)\nu_r(\omega) \\
    & \phantom{:=}\;+ \1_{[\tau(X(\omega)^0),T]}(r)\Big[\nu^1_r(\omega)\1_{\Lambda}(X(\omega)^0) + \nu^2_r(\omega)\1_{\Lambda^c}(X(\omega)^0)\Big]
  \end{align*}
  and let $P_*:=P(0,x,\bar{\nu})\in \cP_x$. We show that $P_*$ satisfies the three properties from~\eqref{eq:PastingCharact}.
  Indeed, $\nu=\bar{\nu}$ on $[0,\tau(X^0))$ implies that $P_*=P$ on $\cF_\tau$. Moreover, as in~\eqref{eq:shiftedMeasureFormula},
  \[
    P_*^{\tau(\omega),\omega} = P\big(\tau(\omega),x\otimes_0 \omega,\bar{\nu}^{\tau(X(\omega)^0),\comega}\big)\quad \mbox{for}\quad P\ae\;\omega\in\Omega.
  \]
  Similarly as below~\eqref{eq:proofPastingShift}, we also have that
  \[
    \bar{\nu}^{\tau(X(\omega)^0),\comega} = \nu^1\big(\comega\otimes_{\tau(X(\omega)^0)} \cdot\big) = (\nu^1)^{\tau(X(\omega)^0),\comega} \quad \mbox{for}\quad P\ae\;\omega\in \Lambda.
  \]
  Therefore,
  \[
    P_*^{\tau(\omega),\omega}=P\big(\tau(\omega),x\otimes_0 \omega,(\nu^1)^{\tau(X(\omega)^0),\comega}\big)=(P^1)^{\tau(\omega),\omega} \quad \mbox{for}\quad P\ae\;\omega\in \Lambda.
  \]
  An analogous argument establishes the third property from~\eqref{eq:PastingCharact} and we conclude that $\bar{P}=P_*\in\cP_x$.
\end{proof}

The second property is the quasi-sure representation of $V^x_t(\xi)$ on $\L^1_{\cP_x}$, a result which will be generalized in Theorem~\ref{th:DPPstop} below.

\begin{lemma}\label{th:DPPesssupL1}
  Let $t\in[0,T]$ and $\xi\in\L^1_{\cP_x}$. Then
  \begin{equation}\label{eq:DPPesssupSpecialL1}
    V^x_t(\xi) = {\mathop{\esssup^P}_{P'\in \cP_x(\cF_t,P)}} E^{P'}[\xi^x|\cF_t]\quad P\as\quad \mbox{for all}\quad  P\in\cP_x,
  \end{equation}
  where $\cP_x(\cF_t,P):=\{P'\in\cP_x:\,P'=P\mbox{ on }\cF_t\}$.
\end{lemma}

\begin{proof}
  Recall that Lemma~\ref{le:filtrGenFollows} allows us to appeal to the results of Section~\ref{se:dynamicProg}.

  (i)~We first prove the inequality ``$\leq$'' for $\xi\in\UC_b(\Omega)$.
  Fix $P\in\cP_x$. We use Step~(ii) of the proof of Theorem~\ref{thm:DPP}, in particular~\eqref{eq:proofDPPConcatMeas}, for the special case $s=0$ and obtain that for given $\eps>0$ and $N\geq1$ there exists a measure $\bar{P}_N\in\cP_x(\cF_t,P)$ such that
  \[
    V^x_t(\omega)\leq E^{\bar{P}_N}[\xi^x| \cF_t](\omega) + \rho(\eps)\quad \mbox{for}\quad P\as\;\omega\in E^1\cup\dots\cup E^N,
  \]
  where $V^x_t(\omega):=V^x_t(\xi;\omega)$.
  Since $\bigcup_{i\geq1} E^i=\Omega^0=\Omega$ $P$-a.s., we deduce that
  \[
    V^x_t(\omega)\leq \sup_{N\geq 1}E^{\bar{P}_N}[\xi^x| \cF_t](\omega) + \rho(\eps)\quad \mbox{for}\quad P\as\;\omega\in \Omega.
  \]
  The claim follows by letting $\eps\to0$.

  (ii)~Next, we show the inequality ``$\geq$'' in~\eqref{eq:DPPesssupSpecialL1} for $\xi\in\UC_b(\Omega)$.  Fix $P,P'\in\cP_x$ and recall that $(P')^{t,\omega}\in\cP(t,x\otimes_0\omega)$ for $P'$-a.s.\ $\omega\in\Omega$
  by Lemma~\ref{le:shiftedMeasure}. Therefore, \eqref{eq:DPP} applied with $s:=t$ and $t:=T$ yields that
  \[
    V^x_t(\omega)=V_t(x\otimes_0\omega) \geq E^{(P')^{t,\omega}}[\xi^{t,x\otimes_0\omega}]= E^{(P')^{t,\omega}}[(\xi^x)^{t,\omega}]=E^{P'}[\xi^x|\cF_s](\omega)
  \]
  $P'$-a.s.\ on $\cF_t$. If $P'\in\cP_x(\cF_t,P)$, then $P'=P$ on $\cF_t$ and the inequality holds also $P$-a.s. The claim follows as $P'\in \cP_x(\cF_t,P)$ was arbitrary.

  (iii) So far, we have proved the result for $\xi\in\UC_b(\Omega)$. The general case $\xi\in\L^1_{\cP_x}$ can be derived by an approximation argument exploiting the stability under pasting (Lemma~\ref{le:stabilityUnderPasting}). We omit the details since the proof is exactly the same as in~\cite[Theorem~5.4]{Nutz.10Gexp}.
\end{proof}

\section{Path Regularity for the Value Process}\label{se:modification}

In this section, we construct a c\`adl\`ag $\cP_x$-modification for $V^x(\xi)$; that is, a  c\`adl\`ag process $Y^x$ such that $Y^x_t=V^x_t(\xi)$ $\cP_x$-q.s.\ for all $t\in [0,T]$. (Recall that the initial condition $x\in\R^d$ has been fixed.)
To this end, we extend the raw filtration $\F$ as in~\cite{NutzSoner.10}: we let
$\F^+=\{\cF_{t+}\}_{0\leq t\leq T}$ be the minimal right-continuous filtration containing $\F$
and we augment $\F^+$ by the collection
$\cN^{\cP_x}$ of $(\cP_x,\cF_T)$-polar sets to obtain the filtration
\[
  \G=\{\cG_t\}_{0\leq t\leq T},\quad \cG_t:= \cF_{t+}\vee\cN^{\cP_x}.
\]
We note that $\G$ depends on $x\in\R^d$ since $\cN^{\cP_x}$ does, but for brevity, we shall not indicate this in the notation. In fact, the dependence on $x$ is not crucial: we could also work with $\F^+$, at the expense of obtaining a modification which is $\cP_x$-q.s.\ equal to a c\`adl\`ag process rather than being c\`adl\`ag itself.

We recall that in the quasi-sure setting, value processes similar to the one under consideration do not admit c\`adl\`ag modifications in general; indeed, while the right limit exists quasi-surely, it need not be a modification (cf.\ \cite{NutzSoner.10}). Both the regularity of $\xi\in \L^1_{\cP_x}$ and the regularity induced by the SDE are crucial for the following result.

\begin{theorem}\label{th:modification}
  Let $\xi\in\L^1_{\cP_x}$. There exists a $($$\cP_x$-q.s.\ unique$)$ $\G$-adapted c\`adl\`ag $\cP_x$-modification $\cE^x(\xi)=\{\cE^x_t(\xi)\}_{t\in[0,T]}$ of $\{V^x_t(\xi)\}_{t\in[0,T]}$. Moreover,
  \begin{equation}\label{eq:Yesssup}
      \cE^x_t(\xi) = \mathop{\esssup^P}_{P'\in \cP_x(\cG_t,P)} E^{P'}[\xi^x|\cG_t]\quad P\as\quad \mbox{for all}\quad P\in\cP_x,
  \end{equation}
  for all $t\in[0,T]$.
\end{theorem}

\begin{proof}
  In view of Lemmata~\ref{le:stabilityUnderPasting} and~\ref{th:DPPesssupL1}, we obtain exactly as in
  \cite[Proposition~4.5]{NutzSoner.10} that there exists a $\cP_x$-q.s.\ unique $\G$-adapted c\`adl\`ag process $\cE^x(\xi)$ satisfying~\eqref{eq:Yesssup} and
  \begin{equation}\label{eq:rightLimit}
    \cE^x_t(\xi)=V^x_{t+}(\xi):=\lim_{r\downarrow t} V^x_r(\xi)\quad\cP_x\qs \quad \mbox{for all}\quad  0\leq t<T.
  \end{equation}
  The observation made there is that~\eqref{eq:DPPesssupSpecialL1} implies that $V^x(\xi)$ is a $(P,\F)$-supermartingale for all $P\in\cP_x$, so that one can use the  standard modification argument for supermartingales under each $P$. This argument, cf.\ \cite[Theorem~VI.2]{DellacherieMeyer.82}, also yields that
  $E^P[\cE^x_t(\xi)|\cF_{t+}]\leq V^x_t(\xi)$ $P$-a.s.\ and in particular
  \[
    E^P[\cE^x_t(\xi)]\leq E^P[V^x_t(\xi)]\quad \mbox{for all}\quad P\in\cP_x.
  \]
  Hence, it remains to show that
  \begin{equation}\label{eq:modificationIneq}
    \cE^x_t(\xi)\geq V^x_t(\xi)\quad\cP_x\qs
  \end{equation}
  for $t\in[0,T)$, which is the part that is known to fail in a more general setting. We give the proof in several steps.

  (i) We first show that $\cE^x_t$ maps $\UC_b(\Omega)$ to $\L^1_{\cP_x}(\cF_t)$, and in fact even to $\UC_b(\Omega_t)$ if a suitable representative is chosen. Let $\xi\in\UC_b(\Omega)$, $r\in(t,T]$ and set $V^x_r:=V_r^x(\xi)$. By Proposition~\ref{pr:valueFunctionCont}, there exists a modulus of continuity $\rho$ independent of $r$ such that
  \[
    |V^x_r(\omega)-V^x_r(\omega')| \leq \rho(\|\omega-\omega'\|_r).
  \]
  Hence the $\cP_x$-q.s.\  limit from~\eqref{eq:rightLimit}
  satisfies
  \[
    |V_{t+}(\omega)-V_{t+}(\omega')| \leq \rho(\|\omega-\omega'\|_r)  \quad\mbox{for all}\quad r\in(t,T]\cap\Q,\quad \cP_x\qs
  \]
  Since $\cE^x_t(\xi)=V_{t+}(\xi),$ taking the limit $r\downarrow t$ yields that
  \[
    |\cE^x_t(\xi)(\omega)-\cE^x_t(\xi)(\omega')| \leq \rho(\|\omega-\omega'\|_t)\quad \cP_x\qs
  \]
  By a variant of Tietze's extension theorem, cf.~\cite{Mandelkern.90}, this implies that $\cE^x_t(\xi)$ coincides $\cP_x$-q.s.\ with an element of $\UC_b(\Omega_t)$. In particular, $\cE^x_t(\xi)\in \L^1_{\cP_x}(\cF_t)$.

  (ii) Next, we show that $\cE_t^x$ is Lipschitz-continuous. Let $\xi,\psi \in\L^1_{\cP_x}$ and $t_n\downarrow t$.
  Using~\eqref{eq:rightLimit}, Fatou's lemma and Lemma~\ref{le:LipschitzAndExtension}, we obtain that
  \begin{align*}
    \|\cE^x_t(\xi)-\cE^x_t(\psi)\|_{L^1_{\cP_x}}
    & =  \big\|\textstyle{\lim_n} |V^x_{t_n}(\xi)- V^x_{t_n}(\psi)|\big\|_{L^1_{\cP_x}}\\
    & = \sup_{P\in \cP_x} E^P\big[\textstyle{\lim_n} |V^x_{t_n}(\xi)- V^x_{t_n}(\psi)|\big] \\
    & \leq \sup_{P\in \cP_x} \liminf_n E^P\big[ |V^x_{t_n}(\xi)- V^x_{t_n}(\psi)|\big] \\
    & \leq \|\xi-\psi\|_{L^1_{\cP_x}}.
  \end{align*}

  (iii) Let $\xi\in\L^1_{\cP_x}$. Then there exist $\xi^n\in\UC_b(\Omega)$ such that $\xi^n\to\xi$ in $L^1_{\cP_x}$ and in thus
  $\cE^x_t(\xi^n)\to\cE^x_t(\xi)$ in $L^1_{\cP_x}$ by Step~(ii). Since $\cE^x_t(\xi^n)\in\L^1_{\cP_x}(\cF_t)$ by Step~(i) and since $\L^1_{\cP_x}(\cF_t)$ is closed in $L^1_{\cP_x}$, we conclude that
  \[
    \cE^x_t(\xi)\in \L^1_{\cP_x}(\cF_t)\quad\mbox{for all}\quad\xi\in \L^1_{\cP_x}.
  \]

  (iv) Let $\xi\in\L^1_{\cP_x}$. Since $V^x_t$ is the identity on $\L^1_{\cP_x}(\cF_t)$, Step~(iii) implies that $\cE^x_t(\xi)=V^x_t(\cE^x_t(\xi))$.
  Moreover, the representations~\eqref{eq:DPPesssupSpecialL1} and~\eqref{eq:Yesssup} yield
  \begin{align*}
    V^x_t(\cE^x_t(\xi))
    & = \mathop{\esssup^P}_{P'\in \cP_x(\cF_t,P)} E^{P'}[\cE^x_t(\xi)|\cF_t] \\
    & \geq \mathop{\esssup^P}_{P'\in \cP_x(\cF_t,P)} E^{P'}\big[E^{P'}[\xi^x|\cG_t]\big|\cF_t\big] \\
    & =  \mathop{\esssup^P}_{P'\in \cP_x(\cF_t,P)} E^{P'}[\xi^x|\cF_t] \\
    & = V^x_t(\xi)\quad P\as\quad \mbox{for all}\quad P\in\cP_x.
  \end{align*}
  We conclude that \eqref{eq:modificationIneq} holds true.
\end{proof}

Since $\cE^x(\xi)$ is a c\`adl\`ag process, its value $\cE^x_\tau(\xi)$ at a stopping time $\tau$ is well defined. The following result states the quasi-sure representation of $\cE^x_\tau(\xi)$ and the quasi-sure version of the dynamic programming principle in its final form.

\begin{theorem}\label{th:DPPstop}
  Let $0\leq \varrho\leq \tau\leq T$ be $\G$-stopping times and $\xi\in\L^1_{\cP_x}$. Then
  \begin{equation}\label{eq:DPPStop}
    \cE^x_\varrho(\xi) = \mathop{\esssup^P}_{P'\in \cP_x(\cG_\varrho,P)} E^{P'}[\cE^x_\tau(\xi)|\cG_\varrho]\quad P\as\quad \mbox{for all}\quad P\in\cP_x
  \end{equation}
  and in particular
  \begin{equation*}%
    \cE^x_\varrho(\xi) = \mathop{\esssup^P}_{P'\in \cP_x(\cG_\varrho,P)} E^{P'}[\xi^x|\cG_\varrho]\quad P\as\quad \mbox{for all}\quad P\in\cP_x.
  \end{equation*}
  Moreover, there exists for each $P\in\cP_x$ a sequence $P_n\in\cP_x(\cG_\varrho,P)$ such that
  \begin{equation*}%
    \cE^x_\varrho(\xi) = \lim_{n\to\infty} E^{P_n}[\xi^x|\cG_\varrho]\quad P\as
  \end{equation*}
  with a $P$-a.s.\ increasing limit.
\end{theorem}

\begin{proof}
  In view of Lemmata~\ref{le:stabilityUnderPasting} and~\ref{th:DPPesssupL1}, the result is derived exactly as in \cite[Theorem~4.9]{NutzSoner.10}.
\end{proof}

As in~\eqref{eq:DPPSemiGrp}, the relation \eqref{eq:DPPStop} can be seen as a semigroup property
\begin{equation*}%
  \cE^x_\varrho(\xi)=\cE^x_\varrho(\cE^x_\tau(\xi)),
\end{equation*}
at least when $\cE^x_\tau(\xi)$ is in the domain $\L^1_{\cP_x}$ of $\cE^x_\varrho$. The latter is guaranteed by Lemma~\ref{le:LipschitzAndExtension} when $\tau$ is a deterministic time. However, one cannot expect $\cE^x_\tau(\xi)$ to be quasi uniformly continuous (cf.\ Remark~\ref{rk:quasiUnifCont}) for a general stopping time, for which reason we prefer to express the right hand side as in~\eqref{eq:DPPStop}.

\section{Hamilton-Jacobi-Bellman 2BSDE}\label{se:2BSDE}

In this section, we characterize the value process $\cE^x(\xi)$ as the solution of a 2BSDE.
To this end, we first examine the properties of $B$ under a fixed $P\in\cP_x$. The following result is in the spirit of~\cite[Section~8]{SonerTouziZhang.2010aggreg}.

\begin{proposition}\label{pr:BunderP}
  Let $x\in\R^d$, $\nu\in\cU$ and $P:=P(0,x,\nu)$. There exists
  a progressively measurable transformation $\beta: \Omega\to\Omega$ $($depending on $x,\nu$$)$ such that
  $W:=\beta(B)$ is a $P$-Brownian motion and
  \begin{equation}\label{eq:filtrationGenUnderP}
    \overline{\F}^{P}=\overline{\F^W}^{P}.
  \end{equation}
  Moreover, $B$ is the $P$-a.s.\ unique strong solution of the SDE
  \[
    B=\int_0^\cdot \mu(t,x+B,\nu_t(W))\,dt+\int_0^\cdot \sigma(t,x+B,\nu_t(W))\,dW_t\quad\mbox{under}\quad P.
  \]
\end{proposition}

\begin{proof}
  Let $X:=X(0,x,\nu)$. As in Lemma~\ref{le:shiftedMeasure}, Assumption~\ref{ass:filtrGen} implies the
  the existence of a progressively measurable transformation $\beta: \Omega\to\Omega$ such that
  \begin{equation}\label{eq:betaNuAppend}
    \beta(X^0)=B \quad P_0\as
  \end{equation}
  Let $W:=\beta(B)$. Then
  \[
    (B,X^0)_{P_0}=(\beta(X^0),X^0)_{P_0} = (\beta(B),B)_P=(W,B)_P;
  \]
  i.e., the distribution of $(B,X^0)$ under $P_0$ coincides with the distribution of $(W,B)$ under $P$. In particular, $W$ is a $P$-Brownian motion. Moreover, we have
  $\overline{\F^{X^0}}^{P_0} = \overline{\F^{\beta(X^0)}}^{P_0}$
  by Assumption~\ref{ass:filtrGen} and therefore
  \[
    \overline{\F^B}^{P}=\overline{\F^{\beta(B)}}^{P},
  \]
  which is~\eqref{eq:filtrationGenUnderP}. Note that
  \[
    X^0=X^0(B)=\int \mu(t,x+X^0,\nu_t(B))\,dt+\int \sigma(t,x+X^0,\nu_t(B))\,dB_t
  \]
  under $P_0$. Let $Y$ be the (unique, strong) solution of the analogous SDE
  \[
    Y=\int \mu(t,x+Y,\nu_t(W))\,dt+\int \sigma(t,x+Y,\nu_t(W))\,dW_t\quad\mbox{under}\quad P.
  \]
  Using the definition of $P$ and~\eqref{eq:betaNuAppend}, we have that
  \begin{align*}
    (Y,W)_P
     &= (X^0(W),W)_P \\
     & = (X^0(B),B)_{P_0} \\
     & = (X^0,\beta(X^0))_{P_0} \\
     & = (B,\beta(B))_P\\
     & =(B,W)_P.
  \end{align*}
  In view of~\eqref{eq:filtrationGenUnderP}, it follows that $Y=B$ holds $P$-a.s.
\end{proof}

In the sequel, we denote by $M^{B,P}$ the local martingale part in the canonical semimartingale decomposition of $B$ under $P$.

\begin{corollary}\label{co:PRP}
  Let $P\in\cP_x$. Then the filtration $\overline{\F}^P$ is right-continuous. If, in addition, $\sigma$ is invertible, then
  $(M^{B,P},P)$ has the predictable representation property.
\end{corollary}

The latter statement means that any right-continuous $(\overline{\F}^P,P)$-local martingale $N$ has a representation $N=N_0+\int Z\,dM^{B,P}$ under $P$, for some $\overline{\F}^P$-predictable process $Z$.

\begin{proof}
  We have seen in Proposition~\ref{pr:BunderP} that $\overline{\F}^P$ is generated by a Brownian motion $W$, hence right-continuous, and that $M^{B,P}=\int_0^\cdot \hat{\sigma}_t \,dW_t$ for $\hat{\sigma}_t:=\sigma(t,x+B,\nu_t(W))$, where $\nu\in\cU$. By changing $\hat{\sigma}$ on a $dt\times P$-nullset, we may assume that $\hat{\sigma}$ is $\overline{\F}^P$-predictable. Using the Brownian representation theorem and $W=\int \hat{\sigma}^{-1} \,dM^{B,P}$, we deduce that $M^{B,P}$ has the representation property.
\end{proof}

The following formulation of 2BSDE is, of course, inspired by~\cite{SonerTouziZhang.2010bsde}.

\begin{definition}
  Let $\xi\in L^1_{\cP_x}$ and consider a pair $(Y,Z)$ of processes with values in $\R\times\R^d$ such that
  $Y$ is c\`adl\`ag $\G$-adapted while
  $Z$ is $\G$-predictable and $\int_0^T |Z_s|^2\,d\br{B}_s<\infty$ $\cP_x$-q.s.
  Then $(Y,Z)$ is called a \emph{solution} of the 2BSDE~\eqref{eq:2bsde} if there exists a family $(K^P)_{P\in\cP_x}$ of $\overline{\F}^P$-adapted increasing processes satisfying $E^P[|K^P_T|]<\infty$ such that
  \begin{equation}\label{eq:2bsde}
   Y_t = \xi - \int_t^T Z_s\,dM^{B,P}_s + K_T^P-K_t^P,\quad 0\leq t\leq T,\quad P\as\quad\mbox{for all }P\in\cP_x
  \end{equation}
  and such that the following minimality condition holds for all $0\leq t\leq T$:
  \begin{equation}\label{eq:minimal}
    \mathop{\essinf^P}_{P'\in \cP_x(\cG_t,P)} E^{P'}\big[K_T^{P'}-K_t^{P'}\big|\cG_t\big]=0\quad P\as \quad\mbox{for all }P\in\cP_x.
  \end{equation}
\end{definition}

Moreover, a c\`adl\`ag process $Y$ is said to be \emph{of class $($D,$\cP_x$$)$} if
the family $\{Y_\tau\}_\tau$ is uniformly integrable under $P$ for all $P\in\cP_x$, where $\tau$ runs through all
$\G$-stopping times. The following is our main result.

\begin{theorem}\label{th:2bsde}
  Assume that $\sigma$ is invertible and let $\xi\in\L^1_{\cP_x}$.
  \begin{enumerate}[topsep=3pt, partopsep=0pt, itemsep=1pt,parsep=2pt]
    \item There exists a $(dt\times\cP_x$-q.s.\ unique$)$ $\G$-predictable process $Z^\xi$ such that
     \begin{equation}\label{eq:defZ}
       Z^\xi = \big(d\br{B,B}^P\big)^{-1}\,d\br{\cE^x(\xi),B}^P \quad P\as\quad\mbox{for all}\quad P\in\cP_x.
     \end{equation}
    \item The pair $(\cE^x(\xi),Z^\xi)$ is the minimal solution of the 2BSDE~\eqref{eq:2bsde}; i.e., if $(Y,Z)$ is another solution, then
            $\cE^x(\xi)\leq Y$ $\cP_x$-q.s.
    \item If $(Y,Z)$ is a solution of~\eqref{eq:2bsde} such that $Y$ is of class $($D,$\cP_x$$)$, then
        $(Y,Z)=(\cE^x(\xi),Z^\xi)$.
  \end{enumerate}
  In particular, if $\xi\in\L^p_{\cP_x}$ for some $p>1$, then $(\cE^x(\xi), Z^\xi)$ is the unique solution of~\eqref{eq:2bsde} in the class $($D,$\cP_x$$)$.
\end{theorem}

\begin{proof}
  Given two processes which are (c\`adl\`ag) semimartingales under all $P\in\cP_x$, their
  quadratic covariation can be defined $\cP_x$-q.s.\ by using the integration-by-parts formula and Bichteler's pathwise stochastic integration \cite[Theorem~7.14]{Bichteler.81}; therefore, the right hand side of~\eqref{eq:defZ} can be used as a definition of $Z^\xi$.
  The details of the argument are as in \cite[Proposition~4.10]{NutzSoner.10}.

  Let $P\in\cP_x$. By Proposition~\ref{pr:BunderP}, $B$ is an It\^o process under $P$; in particular, we have $\br{B,S}^P=\br{M^{B,P},S}^P$ $P$-a.s.\ for any
  $P$-semimartingale $S$. The Doob-Meyer theorem under $P$ and Corollary~\ref{co:PRP} then yield the decomposition
  \[
    \cE^x(\xi) = \cE^x_0(\xi) + \int Z^\xi\,dM^{B,P} - K^P\quad P\as
  \]
  and we obtain (ii) and (iii) by following the arguments in \cite[Theorem~4.15]{NutzSoner.10}.
  If $\xi\in\L^p_{\cP_x}$ for some $p\in (1,\infty)$, then $\cE^x(\xi)$ is of class $($D,$\cP_x$$)$ as a consequence of Jensen's inequality (cf.\ \cite[Lemma~4.14]{NutzSoner.10}). Therefore, the last assertion follows from the above.
\end{proof}

We conclude by interpreting the canonical process $B$, seen under the ``set of scenarios'' $\cP_x$, as a model for drift and volatility uncertainty in the Knightian sense.

\begin{remark}\label{rk:randomG}
  Consider the set-valued process
  \[
    \bD_t(\omega):=\big\{\big(\mu(t,\omega,u),\sigma(t,\omega,u)\big):\,u\in U\big\}\subseteq \R^d\times\R^{d\times d}.
  \]
  In view of Proposition~\ref{pr:BunderP}, each $P\in\cP_x$ can be seen as a scenario in which the
  drift and the volatility (of $B$) take values in $\bD$, $P$-a.s. Then, the upper expectation $\cE^x(\xi)$ is the corresponding  worst-case expectation (see~\cite{NutzSoner.10} for a connection to superhedging in finance). Note that $\bD$ is a random process although the coefficients of our controlled SDE are non-random. Indeed, the path-dependence of the SDE translates to an $\omega$-dependence in the weak formulation that we are considering.

  In particular, for $\mu\equiv 0$, we have constructed a sublinear expectation similar to the random $G$-expectation of~\cite{Nutz.10Gexp}. While the latter is defined by specifying a set-valued process like $\bD$ in the first place, we have started here from a controlled SDE under $P_0$. It seems that the present construction is somewhat less technical that the one in~\cite{Nutz.10Gexp}; in particular, we did not work with the process $\hat{a}=d\br{B}_t/dt$ which played an important role in~\cite{SonerTouziZhang.2010dual} and~\cite{Nutz.10Gexp}. However, it seems that the Lipschitz conditions on $\mu$ and $\sigma$ are essential, while~\cite{Nutz.10Gexp} merely used a notion of uniform continuity.
\end{remark}

\newcommand{\dummy}[1]{}

\end{document}